\documentclass[11pt]{amsart}
\usepackage{graphicx}
\usepackage{amsmath, amsfonts, amsthm}
\usepackage{shuffle}
\usepackage{color}
\numberwithin{equation}{section}

\newcommand{\comm}[1]{}
\newtheorem{theorem}{Theorem}
\newtheorem{definition}[theorem]{Definition}
\newtheorem{lemma}[theorem]{Lemma}
\newtheorem{question}[theorem]{Question}
\newtheorem{remark}[theorem]{Remark}
\newtheorem{proposition}[theorem]{Proposition}
\newtheorem{corollary}[theorem]{Corollary}
\newtheorem{example}[theorem]{Example}

\numberwithin{theorem}{section}
\newtheorem*{acknowledgement}{Acknowledgement}

\theoremstyle{remark}

\DeclareMathOperator{\SU}{SU}

\DeclareMathOperator{\ad}{ad}
\DeclareMathOperator{\Jac}{Jac}
\DeclareMathOperator{\vol}{vol}

\DeclareMathOperator{\tr}{tr}
\DeclareMathOperator{\rtr}{\overline{tr}}
\DeclareMathOperator{\diam}{diam}

\newcommand{\Lie}[1]{\mathfrak{\lowercase{#1}}}

\newcommand{\su}{\Lie{su}}

\newcommand{\R}{\mathbb{R}}
\newcommand{\fg}{\Lie{g}}
\newcommand{\cP}{\mathcal{P}}
\newcommand{\bS}{\mathbf{S}}
\newcommand{\bbS}{\mathbb{S}}
\newcommand{\cH}{\mathcal{H}}
\newcommand{\var}{\text{var}}

\newcommand{\cB}{\mathcal{B}}
\newcommand{\Geod}{\text{Geod}}
\newcommand{\GeodSel}{\text{GeodSel}}

\title{Average signature of geodesic paths in compact Lie groups}
\author{Chong Liu and Shi Wang}

\address{Institute of Mathematical Sciences,
	ShanghaiTech University, Pudong, Shanghai, China}
\email{liuchong@shanghaitech.edu.cn}
\email{wangshi@shanghaitech.edu.cn}

\date{\today}
\keywords{Average signature, trace spectrum, compact Lie group}
\subjclass[2020]{Primary 60L10; Secondary 22E15}

\begin{document}
\maketitle

\begin{abstract}
    For any compact  connected Lie group $G$, we introduce a novel notion of average signature $\mathbb A(G)$ valued in its tensor Lie algebra, by taking the average value of the signature of the unique length-minimizing geodesics between all pairs of generic points in $G$.  we prove that using the average signature together with the trace operation with respect to the given bi-invariant Riemannian metric on $G$, one can recover certain geometric quantities of $G$, including the dimension, the diameter, the volume and the scalar curvature.
\end{abstract}

\section{Introduction}
In a series of seminal papers (\cite{Chen54}, \cite{Chen58}, \cite{Chen73}), K.T. Chen introduced a notion called \textit{signature} (also called signature mapping, signature transform)  which maps smooth curves/paths in a $n$-dimensional differentiable manifold $M$ into the tensor algebra over $\R^n$. More precisely, given a (piecewise) smooth path $\gamma: [0,1] \to M$ and fix a family of smooth one-forms $\phi = \{\phi_1. \ldots, \phi_n\} $ on $M$, then its ($\phi-$) signature is the following formal tensor series
$$
\bS_\phi(\gamma) =  \sum_{k=0}^\infty\sum_{i_1,\ldots,i_k = 1}^n \int_{0<t_1<\ldots<t_k<1} \phi_{i_1}(d \gamma_{t_1}) \ldots \phi_{i_k}(d \gamma_{t_k}) e_{i_1} \otimes \ldots \otimes e_{i_k}
$$
of iterated path integrals, where $e_1, \ldots, e_n$ are the canonical basis of $\R^n$. In the linear case that $M = \R^n$, the classical choice of the one-forms are $\phi_i = dx^i$ for $i=1,\ldots, n$, and in this case the signature of $\R^n$-valued (piecewise) smooth path $\gamma = (\gamma^1, \ldots, \gamma^n)$ can be expressed as
$$
\bS(\gamma) =\sum_{k=0}^\infty\sum_{i_1,\ldots,i_k = 1}^n \int_{0<t_1<\ldots<t_k<1}  (\gamma^{i_1})^\prime(t_1) \ldots (\gamma^{i_k})^\prime(t_k) d t_1\ldots d t_k e_{i_1} \otimes \ldots \otimes e_{i_k}.
$$
As shown in \cite{Chen54} and \cite{Chen58}, the signature mapping is actually an \textit{faithful} group homomorphism (up to the so-called tree-like equivalence, which was later discovered by B. Hambly and T. Lyons in \cite{hambly2010uniqueness}) from the set of (piecewise) smooth paths $\cP$ (equipped with the concatenation of paths as the multiplication) into the completed tensor algebra $T((\R^n))$ (equipped with the tensor product). This result was further generalized to paths with bounded variation (\cite{hambly2010uniqueness}) and to rough paths with low regularity (\cite{BGLY16}). Hence, all information of an unparametrized path $\gamma$ can be encoded into its signature $\bS(\gamma)$, so that studying the tensor $\bS(\gamma)$ is equivalent to studying the path $\gamma$, and then one can take advantage of the rich algebraic structure of the tensor algebra to analyze the signature $\bS(\gamma)$ instead of working directly on the complicated path space $\cP$. For instance, thanks to the following shuffle identity for signature (see e.g. \cite{Freeliealgebras})
$$
\langle \ell_1, \bS(\gamma) \rangle \langle \ell_2, \bS(\gamma) \rangle = \langle \ell_1 \shuffle \ell_2, \bS(\gamma) \rangle 
$$
(for any tensor $\ell_1, \ell_2 \in T(\R^n)$), the Stone-Weierstrass theorem guarantees that the set of linear functionals composition with the signature mappings $\{\langle \ell, \bS(\cdot) \rangle: \ell \in T(\R^n)\}$ is dense in the space of continuous functions on the path space $\cP$, which is often called the universal approximation theorem for the signature, see e.g. \cite{Chevyrev2022signaturekernel} or \cite{LLN2013}.
As a consequence, the \textit{expected signature} can characterize finite Borel measures on the path space $\cP$: for any two  $\mu_1, \mu_2$ finite Borel measures on $\cP$,
$$
\mu_1 = \mu_2 \iff \int_{\gamma \in \cP} \bS(\gamma) d\mu_1(\gamma) = \int_{\gamma \in \cP} \bS(\gamma) d\mu_2(\gamma)
$$
provided both integrals are well defined and satisfy some additional analytical properties, see \cite{Chevyrev2016chf} and \cite{Chevyrev2022signaturekernel}. These nice properties make signature a feasible and powerful tool in many application areas such as statistics and machine learning \cite{kiraly2019kernels}, \cite{CK2016}, \cite{Chevyrev2022signaturekernel}, \cite{Lemercier2021}, \cite{lou2023pcf-gan}, \cite{Cass2021general}; mathematical finance and stochastic control \cite{Kalsi2020Optimal}, \cite{Bayer2021stopping}, stochastic analysis \cite{Cass2024}, \cite{Friz2022magnus} among others.

One of the most significant and challenging open problems regarding signature is the quantitative reconstruction of a path from its signature, which is called the \textit{signature inversion problem}. For instance, it was conjectured in \cite{hambly2010uniqueness} that for any path $\gamma$ in $\R^n$ with bounded variation, its length $L(\gamma)$ can be recovered by considering the following asymptotic formula in terms of the projective norm of the tails of its signature $\bS(\gamma)$:
\begin{equation}\label{eq: length conjecture}
   L(\gamma) = \lim_{k \to \infty} \bigg(k!\| \bS_k(\gamma) \|\bigg)^{\frac{1}{k}}. 
\end{equation}
where $\bS_k(\gamma) = \sum_{i_1,\ldots,i_k = 1}^n \int_{0<t_1<\ldots<t_k<1} d\gamma^{i_1}_{t_1} \ldots d\gamma^{i_k}_{t_k} e_{i_1} \otimes \ldots \otimes e_{i_k} \in (\R^n)^{\otimes k}$ is the $k$-th component of $\bS(\gamma)$.
This conjecture was partially answered for bounded variation paths with various additional assumptions in \cite{hambly2010uniqueness}, \cite{LyonsXu2015inversion}, \cite{Geng2022SLinversion}, \cite{Chang2018inversion} and \cite{Cass2023signature}. Other important contributions in this field include \cite{LyonsXu2018inversion} and \cite{Geng17}.

Very recently, another version of the signature inversion problem has been raised by \cite{Geng2024}: Instead of reconstructing a path from its signature, one wants to recover some geometric features of a given Riemannian manifold $M$ by using the ``average value'' of the signatures of a bunch of special paths which travel therein.  More explicitly, this general problem can be formulated as follows (see also \cite{Geng2024}):

\begin{question}\label{general question}
 Let $M$ be a given Riemannian manifold with a Riemannian metric $g$. For every pair of points $(x,y) \in M \times M$, a curve $\gamma_{x,y} \in \cP(M)$ joining $x$ and $y$ from the set $\cP(M)$ consisting of continuous paths taking values in $M$ is selected and denote its signature (induced by a specified family of one-forms) by $\bbS(\gamma_{x,y})$. If we know the ``average value'' of the signatures of all these paths, denoted as $\mathbb{A}(M)$ and called \textit{``average signature''}, can one explicitly recover some geometric properties (such as curvature properties) from $\mathbb{A}(M)$ ?
\end{question}

In the present paper, we study the above inversion problem by 
focusing on a particular case that $M = G$, where $G$ is a connected, compact Lie group, endowed with a bi-invariant Riemannian metric $\langle \cdot, \cdot \rangle_G$ and the Haar measure $\mu$ induced by the volume form of the metric. 
The main purpose of this paper is to solve the following questions: 

\begin{question}
What is the average signature of length-minimizing geodesic paths in $G$?
\end{question}

As we will show in section \ref{sec:average-signature}, the above question is well posed since given a generic pair of $g,h\in G$ \footnote{Here we mean that for any $g \in G$, $h$ lies outside the cut locus of $g$; see Section \ref{sec:average-signature} for the detailed explanations.}, there is a \emph{unique} length-minimizing geodesic $\gamma_{g,h}$ connecting $g$ and $h$. Thus we may define the average signature on $G$ (with respect to a specified Haar measure $\mu$ on $G$) by
\[\mathbb A(G)=\frac{1}{\mu(G)\times \mu(G)}\int_{G\times G} \mathbb S(\gamma_{g,h}) d\mu(g)d\mu(h).\]
Based on this Lie group invariant, we compute explicit examples (See Example \ref{ex:circle}, \ref{ex:3sphere}, and \ref{ex:torus}) and deduce certain properties (See Proposition \ref{prop:even-degree} and Theorem \ref{thm:product}). Our aim is to study the following question:

\begin{question}\label{question: recover geometric information}
    Can we recover any geometric information of $G$ (and its Haar measure $\mu$) from its average signature $\mathbb A(G)$ explicitly?
\end{question}

In fact, since the geodesics related to any bi-invariant metric on a compact Lie group are same and the Haar measure on a compact Lie group are unique up to a constant, it is not hard to see from the above definition that the average signature is invariant under the choice of the bi-invariant metric on $G$ (see Section 3), and consequently the average signature itself cannot fully recover the Riemannian geometric information of $G$. Hence, the answer to Question \ref{question: recover geometric information} seems to be negative unless $G$ is simple. However, if we are allowed to combine the average signature with the trace operation by taking the contractions on each tensors with respect to a given bi-invariant metric on $G$ (see Definition \ref{def:tr}, \ref{def:rtr}), 
then we can extract useful scalar information from the trace spectrum of average signature and prove the following main theorem(See Theorem \ref{thm:Lk-norm}, \ref{thm:metric-ball} and Corollary \ref{cor:diam}, \ref{cor:dim-vol-scalar}), which confirms that the additional information provided by taking the trace spectrum with respect to the Riemannian metric is enough for us to recover the geometry of $G$ from $\mathbb A (G)$.

\begin{theorem}\label{thm: main result}
     Let $G$ be a connected, compact Lie group with a bi-invariant Riemannian metric. Denote $\overline{\mu}$ the normalized Haar measure of $G$ such that $\overline{\mu}(G)=1$. Given the trace spectrum of the average signature $\tr(\mathbb A(G))\in \mathbb R^\infty$, then we can recover 
\begin{enumerate}
    \item[$\bullet$] the $L^k$-norm (with respect to $\overline {\mu}$) of the function $f(g)=d(e,g)^2$ for all $k\in \mathbb N^*$, where $d(\cdot,\cdot)$ denotes the Riemannian distance function on $G$,
     \item[$\bullet$] the diameter of $G$,
    \item[$\bullet$] the $\overline \mu$-measure $\overline {\mu}(B(R))$ of metric ball centered at any point $g \in G$  with radius $R$ for each $R\geq 0$,
    \item[$\bullet$] the dimension of $G$,
    \item[$\bullet$] the volume of $G$,
    \item[$\bullet$] the scalar curvature of $G$
\end{enumerate}
in an explicit manner.
\end{theorem}

A bit more precisely, we can show that all above geometric properties of $G$ can be recovered by applying the trace operation on the components of the average signature $\mathbb{A}(G)$ with even degrees. For instance, the diameter of $G$ can be recovered by the following asymptotic formula:
\begin{equation}\label{eq: diam of G formula}
\diam(G) = \sup_{g,h \in G} d(g,h) = \lim_{k \to \infty} \bigg((2k)! \tr(A_{2k}(G)) \bigg)^{\frac{1}{2k}}
\end{equation}
where $A_{2k}(G) \in (\R^n)^{\otimes 2k}$ denotes the component of $\mathbb{A}(G)$ with degree $2k$. One can certainly compare the above formula with the famous ``length identity'' \eqref{eq: length conjecture} in the classical signature inversion problem which recovers the length of a single $\R^n$-valued path from the asymptotic behavior of its signature, and notice that a similar asymptotic behavior of the average signature $\mathbb{A}(G)$ on $G$, see \eqref{eq: diam of G formula}, determines the ``length'' of the Lie group $G$. A key ingredient in our above reconstructions is the fact that the length-minimizing geodesics in a Lie group $G$ with a bi-invariant metric can be transformed to straight lines in its Lie algebra/tangent space $\fg$ (under the exponential mapping on $G$), which in turn implies that the average signature of geodesics on the compact Lie group $G$ has the following simple form:
$$
\mathbb{A}(G) = \sum_{k=0}^\infty \frac{1}{\mu(G)}\int_G \frac{v(g)^{\otimes k}}{k!} d\mu(g)
$$
with $v(g) \in \fg$ is the unique vector in the Lie algebra $\fg$ of $G$ such that $\exp(v(g)) = g$, 
so that we can apply some techniques from classical signature theory (which was designed for paths in linear spaces) to study the (average) signature of $G$-valued paths: for instance, 
the recovery of the information of the Haar probability measure $\bar \mu = \frac{\mu}{\mu(G)}$ on $G$ now corresponds to the classical moment/expected signature problem of the pushforward measure of $\bar \mu$ under the function $g \mapsto (t \mapsto v(g)t)$. 

\medskip

\textbf{Related work:} The construction of the signature $\bbS(\gamma)$ of a smooth path $\gamma$ in $G$ used in the present paper is the same as the one defined in \cite{LeeSig2020}, namely $\bbS(\gamma)$ is obtained by computing the iterated integrals of a linear path $\bar \gamma$ via testing a fixed dual frame $\{\omega_1, \ldots, \omega_n\}$ on $G$ against the (derivatives of) path $\gamma$, which allows us to identify $\bbS(\gamma)$, the signature of a $G$-valued path $\gamma$, with $\bS(\bar \gamma)$, the signature of a $\fg$-valued path $\bar \gamma$ (where $\fg$ is the Lie algebra of $G$, in particular, it is a linear space by forgetting the attached Lie bracket), see \cite[Section 3.2]{LeeSig2020} and Section \ref{signature} below. However, we note that the main purpose of the work \cite{LeeSig2020} is to study the applications of signature of Lie group-valued discrete time series in machine learning, and has nothing to do with the reconstruction of any geometric feature of the underlying Lie group.

To our best knowledge, the only literature regarding the inversion problem \ref{general question} is a recent preprint 
\cite{Geng2024} by Geng, Ni and Wang. In their paper, they exploited the expected signature of the Brownian bridge between every pair of points and expected signature of the Brownian loops based on each point to reconstruct the Riemannian distance function, the Ricci curvature, and the second fundamental form on the underlying compact Riemannian manifold.  Compared with our approach which is quite ``deterministic'', the technique used in \cite{Geng2024} relies heavily on the stochastic analysis on manifolds like the Malliavin-Stroock heat kernel expansion and the stochastic differential equations on manifolds. Moreover, we actually obtain different results from \cite{Geng2024}: the results obtained in the present paper are more ``global'' (that is, we recover the dimension, volume of the metric balls of $G$, and so on) whilst the results in \cite{Geng2024} are more ``local'' (as they recover intrinsic and extrinsic  curvature properties). 

\textbf{Outlook:} Comparing with \cite{Geng2024}, it would be very interesting to combine both approaches together to recover more complicated geometric properties of the underlying manifolds, and we will leave it as a possible future exploration. For example, when $M$ is a compact Riemannian manifold, for any pair of points $x,y\in M$, we can first take the expected signature $\psi(t,x,y)$ as in \cite{Geng2024}, and then take the space average
\[\int_{M\times M}\psi(t,x,y)d(\mu\times\mu)\]
where $d\mu$ is the Riemannian volume form on $M$. It is mysterious to know whether this ``averaged expected signature'' will possibly recover any global geometric/topological invariants of $M$.

\textbf{Organizations of the paper:} In Section \ref{signature} we will introduce the notion of signature of Lie-group valued paths and discuss its preliminary properties. In Section \ref{sec:average-signature} the definition of the average signature against a given Haar measure will be given, and then we will provide explicit formulae for the average signature on some classical Lie groups such as $S^1$ and $\text{SU}(2)$. The main results will be recorded in Section \ref{sec:tr}, where we will show how to use only the scalar terms in the trace of the average signature with even degree to reconstruct geometric features of $G$ explicitly, and prove the main Theorem \ref{thm: main result} therein. Finally, in Section \ref{sec:product} we provide an explicit and simple formula which allows us to compute the average signature on the product of Lie groups, provided we know the counterpart on each one.


\begin{acknowledgement} We would like to thank Bin Gui for helpful discussions. We also thank the anonymous referees for helpful suggestions to improve the paper. This work is partially supported by the National Key R\&D Program of China (No. 2023YFA1010900).
\end{acknowledgement}

\section{Signature on Lie groups}\label{signature}
Throughout this paper, let $G$ be a connected, compact Lie group of dimension $n$, endowed with a bi-invariant Riemannian metric $\langle \cdot,\cdot \rangle_G$, whose volume form gives rise to a Haar measure denoted by $\mu$. The associated Riemannian distance on $G$ will be denoted by $d(\cdot,\cdot)$.

Let $\mathfrak g\cong T_e G$ (the tangent space at the identity element $e \in G$) be the corresponding Lie algebra. Since the Lie group $G$ is assumed to be $n$ dimensional, we know that its Lie algebra $\mathfrak g \cong \R^n$ is isomorphic to the $n$-dimensional Euclidean space $\R^n$ as vector spaces. Following the classical notation from the signature theory, we use $T((\mathfrak g))$ to denote the space of formal power series of tensors over $\fg$, i.e., $T((\fg)) = \prod_{k=0}^\infty \fg^{\otimes k}$, and we will call $T((\fg))$ the completed tensor algebra over $\fg$ (recall that the tensor algebra $T(\fg)$ over $\fg$ is the direct sum of tensor powers of $\fg$, i.e., $T(\fg) = \oplus_{k=0}^\infty \fg^{\otimes k}$). More precisely, for any basis $\{e_1, ..., e_n\}$ on $\mathfrak g$, a vector $\mathbf x \in T((\fg))$ can be expressed as
$$
\mathbf x = \sum_{k=0}^\infty\sum_{i_1,\ldots,i_k=1}^n \mathbf x_{i_1,\ldots,i_k} e_{i_1} \otimes \ldots \otimes e_{i_k}
$$
for components $\mathbf x_{i_1,\ldots,i_k} \in \R$. 

Now we fix a basis $\{e_1, ..., e_n\}$ for $\mathfrak g$, it naturally extends to a left invariant frame defined on the entire $G$, which for simplicity we still denote it by $\{e_1, ..., e_n\}$. We denote its unique dual frame on the cotangent bundle by $\{\omega_1,..., \omega_n\}$, that is, $\omega_i(e_j)=\delta_{ij}$ holds at every point in $G$. Clearly, $\{\omega_1,...\omega_n\}$ is a set of smooth differential $1$-form on $G$. Following the same manner in \cite{Chen58} and \cite{LeeSig2020}, we can thus naturally define the signature on the pathspace $\mathcal P(G)$, which consists of continuous and piecewise smooth $C^1$-smooth paths on $G$ defined on the unit interval $[0,1]$, as follows.

\begin{definition}\label{def: signature of Lie group valued path}
The \emph{signature} of $G$-valued (piecewise) $C^1$-smooth paths is a map
\[\mathbb S: \mathcal P(G)\rightarrow T((\mathfrak g))\]
which sends each $C^1$-curve $\gamma:[0,1]\rightarrow G$ to
\[\mathbb S(\gamma)=(S_0(\gamma), S_1(\gamma), \dots, S_k(\gamma),\dots)\]
such that
\[S_0(\gamma) = 1, \quad S_k(\gamma)=\sum_{1\leq i_1,\dots,i_k\leq n}S_{i_1 \cdots i_k}(\gamma)  e_{i_1}\otimes \cdots\otimes e_{i_k} \text{ for } k \ge 1,\]
and
\[S_{i_1 \cdots i_k}(\gamma)=\int_{0<u_1<\cdots<u_k<1}\omega_{i_1}(\gamma'(u_1))\cdots\omega_{i_k}(\gamma'(u_k))du_1\cdots du_k.\]
For convenience, we also write $\mathbb S(\gamma)$ in terms of the formal power series
\[\mathbb S(\gamma)=\sum_{k=0}^\infty S_k(\gamma).\]
\end{definition}

The next proposition shows that the above definition of signature of $G$-valued $C^1$-smooth paths does not rely on the particular choice of basis for the Lie algebra $\fg$.
\begin{proposition}
The map $\mathbb S$ is independent of the choice of the basis $\{e_1,\dots, e_n\}$ on $\mathfrak g$.
\end{proposition}

\begin{proof} The proof is similar to the case of the signature in $\mathbb R^n$. For completeness, we include it here. Suppose $\{\overline{e}_1,\dots,\overline{e}_n\}$ is another basis of $\mathfrak g$ such that
\[\overline{e}_i=\sum_{j=1}^n a_{ij}e_j,\]
where $\{a_{ij}\}$ are constants.
We denote the dual basis $\{\overline \omega_1,\dots,\overline \omega_n\}$. Since both $\{\omega_i\}, \{\overline \omega_i\}$ are left invariant, we may assume
\[\overline\omega_i=\sum_{j=1}^n b_{ij}\omega_j,\]
where $\{b_{ij}\}$ are constants.
Then using the duality relations $\delta_{ij}=\overline \omega_i(\overline e_j)=\omega_i(e_j)$, we have
\begin{equation}\label{eq:dual-basis}
    \delta_{ij}=\sum_{k=1}^n b_{ik}a_{jk}=\sum_{k=1}^n b_{ki}a_{kj}.
\end{equation}
Denote by $\overline{\mathbb S}(\gamma)=(\overline S_0(\gamma), \overline S_1(\gamma),\dots, \overline S_k(\gamma),\dots)$ the signature under the basis $\{\overline{e}_1,\dots,\overline{e}_n\}$, and the corresponding component on $\overline{e}_{i_1}\otimes\cdots\otimes \overline{e}_k$ by $\overline{S}_{i_1\cdots i_k}(\gamma)$. Then by definition,
\begin{align*}
    \overline S_{i_1\cdots i_k}(\gamma)&=\int_{0<u_1<\cdots<u_k<1}\overline\omega_{i_1}(\gamma'(u_1))\cdots\overline\omega_{i_k}(\gamma'(u_k))du_1\cdots du_k\\
    &=\sum_{1\leq j_1,\dots,j_k\leq n}b_{i_1j_1}\cdots b_{i_kj_k}\int_{0<u_1<\cdots<u_k<1}\omega_{j_1}(\gamma'(u_1))\cdots\omega_{j_k}(\gamma'(u_k))du_1\cdots du_k\\
    &=\sum_{1\leq j_1,\dots,j_k\leq n}b_{i_1j_1}\cdots b_{i_kj_k} S_{j_1\cdots j_k}(\gamma).
\end{align*}
Thus we have,
\begin{align*}
    \overline S_n(\gamma)&=\sum_{1\leq i_1\dots, i_k\leq n}\overline S_{i_1\dots i_k}(\gamma)\overline e_{i_1}\otimes\cdots \otimes \overline e_{i_k}\\
    &=\sum_{1\leq i_1\dots, i_k\leq n}\sum_{1\leq j_1,\dots,j_k\leq n}b_{i_1j_1}\cdots b_{i_kj_k} S_{j_1\cdots j_k}(\gamma)\overline e_{i_1}\otimes\cdots \otimes \overline e_{i_k}\\
    &=\sum_{1\leq i_1\dots, i_k\leq n}\sum_{1\leq j_1,\dots,j_k\leq n}\sum_{1\leq l_1\dots, l_k\leq n}b_{i_1j_1}\cdots b_{i_kj_k}a_{i_1l_1}\cdots a_{i_kl_k}S_{j_1\cdots j_k}(\gamma)e_{l_1}\otimes\cdots \otimes e_{l_k}\\
    &=\sum_{1\leq j_1,\dots,j_k\leq n}\sum_{1\leq l_1\dots, l_k\leq n}\delta_{j_1l_1}\cdots \delta_{j_kl_k}S_{j_1\cdots j_k}(\gamma)e_{l_1}\otimes\cdots \otimes e_{l_k}\\
    &=\sum_{1\leq j_1,\dots,j_k\leq n}S_{j_1\cdots j_k}(\gamma)e_{j_1}\otimes\cdots \otimes e_{j_k}\\
    &=S_n(\gamma),
\end{align*}
where the fourth equality uses equation \eqref{eq:dual-basis}. Therefore, the proposition holds.
\end{proof}

In fact, we can interpret the above definition of the signature of $G$-valued  (piecewise) $C^1$-smooth paths in an alternative way: Let $\gamma \in \cP(G)$, we can obtain a $\fg$-valued continuous path $f_\gamma: [0,1] \to \fg$ by defining
$$
f_\gamma (t) := (L_{\gamma(t)^{-1}})_*(\gamma^\prime(t)) := d(L_{\gamma(t)^{-1}})_{\gamma(t)} \gamma^\prime (t) \in T_e G \cong \fg,
$$
where $L_g$ denotes the left translation by $g \in G$, $g^{-1}$ denotes the inverse of $g \in G$, and $(L_g)_*: T_h G \to T_{gh} G$ denotes the push-forward operation. Note that this path $f_\gamma$ is the pullback of the Maurer-Cartan form along the path $\gamma$ (see \cite[Section 5.6]{DGLee}).

Now let $\bar \gamma \in \cP(\fg)$ be a $C^1$-smooth path valued in $\fg$ which is defined by
$$
\bar \gamma(t) = \int_0^t f_\gamma (s) ds, \quad t \in [0,1].
$$
In other words, $\bar \gamma$ is the unique curve in $\fg$ (starting from $0 \in \fg$) which satisfies $\bar \gamma^\prime (t) = f_\gamma(t)$ for all $t \in [0,1]$.
For a fixed basis $\{e_1,\ldots,e_n\}$ of $\fg$, we can express $\bar \gamma(t)$ as
$$
\bar \gamma(t) = \sum_{i=1}^n \bar \gamma^i(t) e_i,
$$
where the components $\bar \gamma^i$, $i=1,\ldots,n$ are $\R$-valued  $C^1$-smooth paths. Hence, we may view $\bar \gamma$ as an $\R^n$-valued $C^1$-smooth path. This allows us to apply the classical notion of signature for $C^1$-smooth $\R^n$-valued paths (see e.g. \cite{hambly2010uniqueness}) to define the signature of path $\bar \gamma$ 
$$
\bS(\bar \gamma) = \sum_{k=0}^\infty \bS_k(\bar \gamma),\quad \bS_k(\bar \gamma) = \sum_{i_1,\ldots,i_k=1}^n \bS_{i_1,\ldots,i_k}(\bar \gamma)e_{i_1} \otimes \ldots \otimes e_{i_k},
$$
where $\bS_{i_1,\ldots,i_k}(\bar \gamma)  =\int_{0<u_1<\cdots<u_k<1}(\bar \gamma^{i_1})^\prime (u_1)\cdots (\bar \gamma^{i_k})^\prime (u_k)du_1\cdots du_k$. \\
On the other hand, recalling that $\{\omega_1, \ldots, \omega_n\}$ are the left invariant dual frame associated with $\{e_1,\ldots,e_n\}$ which satisfies that 
$$
\omega_i (\gamma^\prime (t)) = e_i^* (d(L_{\gamma(t)^{-1}})_{\gamma(t)} \gamma^\prime (t) ) = e_i^*(f_\gamma(t)) = e_i^*(\bar \gamma^\prime(t)) = (\bar \gamma^i)^\prime(t)
$$
(where $\{e_1^*,\ldots,e^*_n\}$ denotes the dual basis of $\{e_1,\ldots,e_n\}$ in $T^*_eG$) for all $t \in [0,1]$ and all $i=1,\ldots,n$, we indeed have
\begin{align*}
   \bS_{i_1,\ldots,i_k}(\bar \gamma)  &=\int_{0<u_1<\cdots<u_k<1}(\bar \gamma^{i_1})^\prime (u_1)\cdots (\bar \gamma^{i_k})^\prime (u_k)du_1\cdots du_k \\
   &=\int_{0<u_1<\cdots<u_k<1}\omega_{i_1}(\gamma'(u_1))\cdots\omega_{i_k}(\gamma'(u_k))du_1\cdots du_k = S_{i_1 \cdots i_k}(\gamma) 
\end{align*}
for all $i_1,\ldots,i_k \in \{1,\ldots,n\}$. Hence, in view of Definition \ref{def: signature of Lie group valued path}, we actually obtain that $\bS(\bar \gamma) = \bbS(\gamma)$ for all $\gamma \in \cP(G)$. For later use, we record the above observation into the following definition and proposition.

\begin{definition}\label{def: the mapping phi}
We use $\Phi: \cP(G) \to \cP(\fg)$ to denote the mapping that transforms (piecewise) smooth paths in the Lie group $G$ to (piecewise) smooth paths in the Lie algebra $\fg$ of $G$ which is defined by
$$
\Phi(\gamma) = \bar \gamma,
$$
where $\bar \gamma(t) = \int_0^t f_\gamma(s) ds$ with $f_\gamma(t) = (L_{\gamma(t)^{-1}})_*(\gamma^\prime(t))$ for all $t \in [0,1]$. 
\end{definition}

\begin{proposition}\label{prop: two signatures are same}
  Let $\Phi: \cP(G) \to \cP(\fg)$ be the mapping defined as in Definition \ref{def: the mapping phi} such that $\Phi(\gamma) \in \cP(\fg)$ is the unique curve in $\fg$ starting from $0 \in \fg$ and satisfies 
  $$
  \Phi(\gamma)^\prime(t)  = (L_{\gamma(t)^{-1}})_*(\gamma^\prime(t)) \in T_e G \cong \fg.
  $$
Let $\bbS(\gamma)$ be the signature of the $G$-valued (piecewise) $C^1$-smooth path $\gamma$ defined as in Definition \ref{def: signature of Lie group valued path} and let $\bS(\Phi(\gamma))$ be the classical signature of the $\fg$-valued (piecewise) $C^1$-smooth path $\Phi( \gamma)$. Then we have
$$
\bbS(\gamma) = \bS(\Phi(\gamma)).
$$
\end{proposition}

By using the mapping $\Phi$ defined in Proposition \ref{prop: two signatures are same}, we can apply many techniques from classical signature theory designed for vector space-valued (piecewise) smooth paths to establish useful results for the signature of $G$-valued (piecewise) smooth paths. For instance, we can prove that the signatures of all (piecewise) smooth paths $\gamma \in \cP(G)$ are uniformly bounded by a constant only depending on the diameter of $G$. 

\begin{definition}\label{def: Hilbert space in tensor algebra}
 Let $\langle \cdot, \cdot \rangle_G$ be a bi-invariant Riemannian metric on $G$, and $\langle \cdot, \cdot \rangle_{\fg}$ be its restriction on $T_e G \cong \fg$, which is an inner product on $\fg$ such that $(\fg,\langle \cdot, \cdot \rangle_{\fg})$ is a Hilbert space. Then for every $k \ge 1$, the tensor space $\fg^{\otimes k}$ admits a natural inner product $\langle \cdot,\cdot \rangle_{\fg^{\otimes k}}$ such that if $\{e_1,\ldots,e_n\}$ is an orthonormal basis in $(\fg,\langle \cdot, \cdot \rangle_{\fg})$, then $\{e_{i_1} \otimes \ldots \otimes e_{i_k}: i_1,\ldots,i_k =1,\ldots,n\}$ is an orthonormal basis with respect to $\langle \cdot,\cdot \rangle_{\fg^{\otimes k}}$. Let $\|\cdot \|_{\fg^{\otimes k}}$ denote the norm induced by $\langle \cdot,\cdot \rangle_{\fg^{\otimes k}}$ on $\fg^{\otimes k}$. We define a Hilbert space  
 $$
 \mathcal H := \bigg\{\mathbf{x} \in T((\fg)): \sum_{k=0}^\infty \|\mathbf{x}^k\|^2_{\fg^{\otimes k}} < \infty  \bigg\} \subset T((\fg)),
 $$
 (where $\mathbf{x}^k \in \fg^{\otimes k}$ denotes the projection of $\mathbf x$ into $\fg^{\otimes k}$) equipped with the inner product
 $$
 \langle \mathbf{x}, \mathbf{y} \rangle_{\mathcal{H}} := \sum_{k=0}^\infty \langle \mathbf{x}^k,\mathbf{y}^k\rangle_{\fg^{\otimes k}}
 $$
 and the Hilbert norm
 $$
 \|\mathbf{x}\|_{\mathcal{H}} := \bigg(\sum_{k=0}^\infty \|\mathbf{x}^k\|^2_{\fg^{\otimes k}}\bigg)^{1/2}.
 $$
\end{definition}

\begin{proposition}\label{prop: boundedness of signature}
For any $\gamma \in \cP(G)$, we have $\bbS(\gamma) \in \mathcal{H}$ and there is a constant $C$ only depending on $\dim G$ and the length of $\gamma$ such that
$$
\|\bbS(\gamma)\|_{\mathcal{H}} \le C.
$$
In particular, for all $\gamma \in \Geod(G)$ where $\Geod(G):= \{\gamma \in \cP(G): \forall s,t \in [0,1],  d(\gamma(s), \gamma(t)) = |t-s|d(\gamma(0), \gamma(1))\}$ be the space of all $C^1$-smooth constant speed shortest geodesics in $G$, the above bound $C$ only depends on $\dim G$ and $\diam(G) = \sup \{d(g,h):g,h  \in G\}$.
\end{proposition}

\begin{proof}
The proof can be found in \cite[Lemma 35]{LeeSig2020}, but for reader's convenience we provide the sketch of the proof here. In view of Proposition \ref{prop: two signatures are same}, we have $\bbS(\gamma) = \bS(\Phi(\gamma))$; on the other hand, for the vector space-valued smooth curve $\Phi(\gamma) \in \fg$, the classical result, see e.g. \cite{LyonsQian2007} or \cite{FrizVictoir2010}, tells us that there is a constant $C$ only depending on the dimension of $\fg$ (which is equal to $n = \dim G$) and the length of $\Phi(\gamma)$ with respect to the inner product $\langle \cdot,\cdot \rangle_{\fg}$ (denoted by $\ell(\Phi(\gamma)) = \int_0^1 \|\Phi(\gamma)^\prime(t)\|_{\fg} dt$) in an increasing way such that
$$
\|\bbS(\gamma)\|_{\cH} = \|\bS(\Phi(\gamma)\|_{\mathcal{H}} \le C = C(n, \ell(\Phi(\gamma)).
$$
Now, we note that since $\langle \cdot,\cdot \rangle_{\fg}$ is the restriction of the bi-invariant metric $\langle \cdot, \cdot \rangle_G$ on $T_e G \cong \fg$, it holds that
\begin{align*}
    \ell(\Phi(\gamma)) &= \int_0^1  \langle \Phi(\gamma)^\prime(t), \Phi(\gamma)^\prime(t) \rangle_{\fg}^{1/2}dt \\
    &=\int_0^1  \langle (L_{\gamma(t)^{-1}})_*(\gamma^\prime(t)), (L_{\gamma(t)^{-1}})_*(\gamma^\prime(t)) \rangle_{\fg}^{1/2}dt\\
    &=\int_0^1  \langle \gamma^\prime(t), \gamma^\prime(t) \rangle_{G}^{1/2}dt\\
    &=\ell(\gamma)
\end{align*}
where $\ell(\gamma)$ denotes the length of $\gamma$ with respect to $\langle \cdot, \cdot \rangle_G$. So, for all $\gamma \in \cP(G)$, we deduce from the above that
$$
\|\bbS(\gamma)\|_{\cH} = \|\bS(\Phi(\gamma)\|_{\mathcal{H}} \le C = C(n, \ell(\gamma)).
$$
If $\gamma \in \Geod(G)$, since $\ell(\gamma) = d(\gamma(0), \gamma(1)) \le \diam (G)$, we have
$\|\bbS(\gamma)\|_{\cH} \le C(n, \diam(G))$.
\end{proof}

Next we will show the stability of the signature map on $\cP(G)$. As in \cite[Section 3.3]{LeeSig2020}, we endow the pathspace $\cP(G)$ with a metric
\begin{equation}\label{eq: metric on pathspace}
    d_{\cP(G)} (\alpha, \beta) := \sup_{t \in [0,1]} \|\Phi(\alpha)^\prime(t) - \Phi(\beta)^\prime(t)\|_{\fg} + d(\alpha(0), \beta(0)),
\end{equation}
where $\Phi: \cP(G) \to \cP(\fg)$ is defined as in Proposition \ref{prop: two signatures are same}.

\begin{proposition}\label{prop: stability of signature}
 The signature map $\bbS: \cP(G) \to \cH$ is locally Lipschitz continuous with respect to the metric $d_{\cP(G)} (\alpha, \beta)$: there is a constant $C$ only depending on $\dim G$ and $\lambda$ such that
 $$
 \|\bbS(\alpha) - \bbS(\beta)\|_{\cH} \le C d_{\cP(G)} (\alpha, \beta)
 $$
 for all $\alpha, \beta \in \cP(G)$ with $\ell(\alpha) \le \lambda$ and $\ell(\beta) \le \lambda$. In particular, on the set $\Geod(G)$ of all constant speed shortest geodesics in $G$, $\bbS: \Geod(G) \to \cH$ is Lipschitz.
\end{proposition}
\begin{proof}
 This stability result was also mentioned in \cite[Corollary 38]{LeeSig2020}  without a concrete proof. Again for reader's convenience we give a short proof here. Since $\Phi(\alpha)$ and $\Phi(\beta)$ are $\fg$-valued smooth paths, the classical stability result for signatures of vector-space valued paths, see e.g. \cite{LyonsQian2007} and \cite{FrizVictoir2010}, shows that there is a constant $C = C(n,\max\{\ell(\Phi(\alpha)), \ell(\Phi(\beta))\})$ such that 
 $$
  \|\bS(\Phi(\alpha)) - \bS(\Phi(\beta))\|_{\cH} \le C \int_0^1 \|\Phi(\alpha)^\prime(t) - \Phi(\beta)^\prime(t)\|_{\fg} dt \le Cd_{\cP(G)} (\alpha, \beta)
 $$
 (In fact, it is easy to see that $\int_0^1 \|\Phi(\alpha)^\prime(t) - \Phi(\beta)^\prime(t)\|_{\fg} dt$ is exactly the bounded variation distance between $\Phi(\alpha)$ and $\Phi(\beta)$, i.e.,
 \begin{align*}
     \|\Phi(\alpha) - \Phi(\beta)\|_{1-\var} &= \sup_{0=t_0<\ldots<t_N=1}\|\Phi(\alpha)_{t_i,t_{i+1}} - \Phi(\beta)_{t_i,t_{i+1}}\|_{\fg} \\
     &= \int_0^1 \|\Phi(\alpha)^\prime(t) - \Phi(\beta)^\prime(t)\|_{\fg} dt
 \end{align*}
 for $\Phi(\alpha)_{t_i,t_{i+1}} := \Phi(\alpha)(t_{i+1}) - \Phi(\alpha)(t_{i})$, and the bounded variation metric is used for measuring the stability of signature mappings in classical signature theory). 
 
 If $\alpha$ and $\beta$ belong to $\Geod(G)$, then by the proof of Proposition \ref{prop: boundedness of signature} we know that $\max\{\ell(\Phi(\alpha)), \ell(\Phi(\beta))\} = \max\{\ell(\alpha), \ell(\beta)\} \le \diam(G)$, therefore, using Proposition \ref{prop: two signatures are same}, we finally get that
 $$
 \|\bbS(\alpha) - \bbS(\beta)\|_{\cH} = \|\bS(\Phi(\alpha)) - \bS(\Phi(\beta))\|_{\cH} \le C(n, \diam(G))d_{\cP(G)} (\alpha, \beta),
 $$
as claimed.
\end{proof}

\begin{remark}\label{remark: important properties of signature of Lie group valued paths}
As we have mentioned before, since the signatures of $G$-valued (piecewise) smooth paths defined as in Definition \ref{def: signature of Lie group valued path} can be identified with the signatures of $\fg$-valued (piecewise) smooth paths under the mapping $\Phi$, all results from the classical signature theory designed for vector space-valued (piecewise) smooth paths remain valid. Besides the above properties contained in Propositions \ref{prop: boundedness of signature} and \ref{prop: stability of signature}, we also have the following observations, most of which can also be found in \cite{LeeSig2020}:
\begin{enumerate}
    \item For two paths $\alpha$ and $\beta$ in $\cP(G)$, $\bbS(\alpha) = \bbS(\beta)$ if and only if $\alpha$ and $\beta$ are tree-like equivalent. For the definition of tree-like equivalence, see e.g. \cite[Definition 1.1]{BGLY16}. As a consequence, if we use $\sim_{tree}$ to denote the tree-like equivalence relation, and let $\widetilde \cP(G) = \cP(G)/\sim_{tree}$ be the quotient space of $\cP(G)$ modulo $\sim_{tree}$, then the signature mapping $\bbS: \widetilde \cP(G) \to \cH$ defined by $\bbS([\alpha]) := \bbS(\alpha)$ is well-defined for any equivalence class $[\alpha] \in \widetilde \cP(G)$ and is injective on $\widetilde \cP(G)$. 
    \item The shuffle identity still holds for the signature mapping $\bbS$ on $\cP(G)$: for any $h_1, h_2  \in T((\fg))$ and any $\alpha \in \cP(G)$, one has
    $$
    \langle h_1, \bbS(\alpha)\langle \langle h_2, \bbS(\alpha) \rangle = \langle h_1 \shuffle h_2, \bbS(\alpha)\rangle, 
    $$
    where $\langle h, \bbS(\alpha) \rangle$ denotes the dual pairing between $h \in T(\fg)$ and $\bbS(\alpha) \in T((\fg))$, and $\shuffle$ denotes the shuffle product (see e.g. \cite{Freeliealgebras}). 
    \item In view of Proposition \ref{prop: stability of signature} and the previous two observations, the set $\{\langle h, \Lambda \circ \bbS(\cdot) \rangle: h \in T(\fg)\}$ is an algebra in $C_b(\widetilde \cP(G))$ which can separate points, where $C_b(\widetilde \cP(G))$ is the family of all continuous and bounded functions on $\widetilde \cP(G)$ equipped with the quotient topology induced by $d_{\cP(G)}$ on $\cP(G)$, and $\Lambda: T((\fg)) \to T((\fg))$ is a tensor normalization (see \cite{Chevyrev2022signaturekernel} for the construction of $\Lambda$) which preserves the shuffle identity and satisfies that $\|\Lambda(x)\|_{\cH} \le 1$ for all $x \in \cH$. Hence, by the Giles' Theorem (see \cite[Theorem 9]{Chevyrev2022signaturekernel} or \cite{Giles71}), the set $\{\langle h, \Lambda \circ \bbS(\cdot) \rangle: h \in T(\fg)\}$ is dense in $C_b(\widetilde \cP(G))$ with respect to the strict topology on $C_b(\widetilde \cP(G))$ (for the definition of strict topology, see \cite{Giles71}). In other words, the normalized signature $\Lambda \circ \bbS$ satisfies the universal approximation property in $C_b(\widetilde \cP(G))$ for the strict topology. As a consequence, the normalized signature $\Lambda \circ \bbS$ is characteristic to the space of finite regular Borel measures on $\widetilde \cP(G)$: for $\nu_1$ and $\nu_2$ two finite regular Borel measures on $\widetilde \cP(G)$, one has $\nu_1 = \nu_2$ if and only if $\int \Lambda \circ \bbS d\nu_1 = \int \Lambda \circ \bbS d\nu_2$, where these integrals are $\cH$-valued Bochner integrals.
    For more details on the above results, we refer to \cite{Chevyrev2022signaturekernel} or \cite[Section 4]{LeeSig2020}. 
    \item Since the signature mapping $\bbS$ (and the normalized signature $\Lambda \circ \bbS$) take values in a Hilbert space $\cH$, they induce a Reproducing Kernel Hilbert Space (for short, RKHS) $\cH_k$ with the kernel $k(\alpha, \beta) := \langle \bbS(\alpha), \bbS(\beta) \rangle_{\cH}$ for $\alpha, \beta \in \cP(G)$. In particular, one can use the so-called signature kernel trick, see \cite{kiraly2019kernels}, \cite{Chevyrev2022signaturekernel} and \cite{Cass2021general}, which was originally designed for vector-space valued paths, to compute the kernel $k(\alpha, \beta)$ efficiently (without having to compute explicit $\bbS(\alpha)$ and $\bbS(\beta)$) for $G$-valued paths $\alpha$ and $\beta$. For instance, the kernel $k(\alpha, \beta)$ is actually the solution $F(1,1)$ for the following Goursat PDE
    $$
    \frac{\partial^2}{\partial s \partial t}F(s,t) = \langle \alpha^\prime(s), \beta^\prime(t) \rangle_G F(s,t), \quad F(0,\cdot) = F(\cdot,0) = 1,
    $$
    see \cite{SalviPDE}. For more details on this kernel learning technique for $G$-valued path signature $\bbS$ and its applications in human action recognition and random walk on Lie groups, please see \cite[Sections 4,5]{LeeSig2020}.
\end{enumerate}
\end{remark}

In order to define the average signature over $\cP(G)$ in the next section, we need the following purely measure-theoretic result.

\begin{proposition}\label{prop: measurable selector}
 Let $\Geod(G):= \{\gamma \in \cP(G): \forall s,t \in [0,1],  d(\gamma(s), \gamma(t)) = |t-s|d(\gamma(0), \gamma(1))\}$ be the space of all $C^1$-smooth constant speed shortest geodesics in $G$.  Then we have
 \begin{enumerate}
     \item $(\Geod(G), d_{\cP(G)})$ is a Polish metric space (i.e., complete and separable), where $d_{\cP(G)}(\cdot,\cdot)$ is defined as in \eqref{eq: metric on pathspace}.
     \item There exists a Borel measurable mapping $\GeodSel: G \times G \to \Geod(G)$ such that for any pair $(g,h) \in G \times G$, $\GeodSel(g,h) \in \Geod(G)$ is a $C^1$-smooth constant speed geodesic with $\GeodSel(g,h)(0) = g$ and $\GeodSel(g,h)(1) = h$, where $G \times G$ is equipped with the product Borel $\sigma$-algebra generated by $d(\cdot,\cdot)$ on $G$ and $\Geod(G)$ is equipped with the Borel $\sigma$-algebra $\cB(d_{\cP(G)})$ generated by the metric $d_{\cP(G)}$.
     \item The composition mapping $\bbS \circ \GeodSel: G \times G \to \cH$ is Borel measurable.
 \end{enumerate}
\end{proposition}

\begin{proof}
\begin{enumerate}
    \item First we show that $(\cP(G), d_{\cP(G)})$ is a Polish metric space. Indeed, let $C([0,1],\fg)$ be the space of continuous paths in $\fg$ equipped with the supremum norm $\|\cdot\|_\infty$ (i.e., for $f, \bar f \in C([0,1],\fg)$, $\|f - \bar f\|_\infty := \sup_{t \in [0,1]}\|f(t) - \bar f(t)\|_{\fg}$), consider the mapping $$
    \bar \Phi: \cP(G) \to G \times C([0,1],\fg), \gamma \mapsto (\gamma(0), \Phi(\gamma)^\prime).
    $$
    Note that for every given $(g, t \mapsto f(t)) \in G \times C([0,1], \fg)$, let  $\gamma \in \cP(G)$ be the unique solution to the ODE
    $$
    \gamma^\prime(t) =F(\gamma(t), t), \quad \gamma(0) = g,
    $$
    where $F(h, t):=(L_h)_* f(t) = d(L_h)_e f(t)$ for $h \in G$ and $t \in [0,1]$, then we indeed have 
    $$
    \Phi(\gamma)^\prime(t) = (L_{\gamma(t)^{-1}})_{*}\gamma^\prime(t) = d(L_{\gamma(t)^{-1}})_{\gamma(t)} d(L_{\gamma(t)})_e f(t) = f(t)
    $$
   for all $t \in [0,1]$, which implies that
   $$
   \bar \Phi (\gamma) = (\gamma(0), \Phi(\gamma)^\prime) = (g, t \mapsto f(t)).
   $$
   So, if we denote $\bar \Phi^{-1}(g, t \mapsto f(t)) = \gamma$, then it holds that $\bar \Phi \circ \bar \Phi^{-1} = \text{id}_{G \times C([0,1],\fg)}$ and $\bar \Phi^{-1} \circ \bar \Phi = \text{id}_{\cP(G)}$. This means exactly that 
   $$
   \bar \Phi: (\cP(G), d_{\cP(G)}) \to (G,d) \times (C([0,1],\fg),\|\cdot\|_\infty) 
   $$
    is a surjective isometry. Then as $(G,d)$ and $(C([0,1],\fg),\|\cdot\|_\infty)$ are obviously Polish metric spaces, so is their product and therefore $(\cP(G), d_{\cP(G)})$ is a Polish metric space as well. \\
    Now we prove that $\Geod(G)$ is closed in $(\cP(G), d_{\cP(G)})$. Suppose that $(\gamma^n)_{n\ge 1}$ is a sequence in $\Geod(G)$ converging to another path $\gamma \in \cP(G)$ with respect to the metric $d_{\cP(G)}(\cdot,\cdot)$, which actually implies the uniform convergence of the vector fields $F^n(h,t):= (L_h)_*\Phi(\gamma^n)^\prime(t)$ to $F(h,t):=(L_h)_*\Phi(\gamma)^\prime(t)$ (in $t$) and the convergence of the initial conditions $\gamma^n(0)$ to $\gamma(0)$. 
    The standard continuity result in the ODE theory then  provides 
    $$
    \lim_{n \to \infty}\sup_{t \in [0,1]} d(\gamma(t), \gamma^n(t)) = 0,
    $$
    as every $\gamma^n$ solves the ODE
    $$
    (\gamma^n)^\prime(t) =  d(L_{\gamma^n(t)})_e \underbrace{d(L_{\gamma^n(t)^{-1}})_{\gamma^n(t)} (\gamma^n)^\prime (t)}_{= \Phi(\gamma^n)^\prime(t)} = F^n(\gamma^n(t), t), 
    $$
   with the initial value $\gamma^n(0)$ and $\gamma$ is the unique solution to the ODE $\gamma^\prime(t) = F(\gamma(t),t)$ with the initial value $\gamma(0)$. Then, since every $\gamma_n \in \Geod(G)$ satisfies 
   $$
   \forall s,t \in [0,1], \quad 
   d(\gamma^n(s), \gamma^n(t)) = |t-s|d(\gamma^n(0), \gamma^n(1)) 
   $$
   the uniform convergence $\lim_{n \to \infty}\sup_{t \in [0,1]} d(\gamma(t), \gamma^n(t)) = 0$ ensures that
   $$
   \forall s,t \in [0,1], \quad 
   d(\gamma(s), \gamma(t)) = |t-s|d(\gamma(0), \gamma(1)),
   $$
   that is, $\gamma \in \Geod(G)$ is a $C^1$-smooth constant speed geodesic, whence the closedness of $\Geod(G)$. It follows that $(\Geod(G), d_{\cP(G)})$ is a Polish metric space.
    \item Since $(G,d)$ is a compact Riemannian manifold, it is a geodesic space, i.e., for any pair $(g,h) \in G \times G$, the set $\Gamma_{g,h} := \{\gamma \in \Geod(G): \gamma(0) = g, \gamma(1) = h\} \subset \Geod(G)$ is nonempty. Moreover, it is easy to show (as we have shown the closedness of $\Geod(G)$ in Step 1) that the graph of the multi-valued map $(g,h) \mapsto \Gamma_{g,h}$ is closed in $G \times G \times \Geod(G)$, where $G$ is equipped with the metric $d$ and $\Geod(G)$ is endowed with the metric $d_{\cP(G)}$. Then as $\Geod(G)$ is a Polish metric space, by the measurable selection theorem (\cite[Theorem 18.26]{infinitedimanalysis}) there exists a Borel measurable selector $\GeodSel: G \times G \to \Geod(G)$ such that $\GeodSel(g,h) \in \Gamma_{g,h}$ is a $C^1$-smooth constant speed shortest geodesic joining $g$ and $h$.
    \item In Proposition \ref{prop: stability of signature} we have proved that the signature map $\bbS: \Geod(G) \to \cH$ is Lipschitz continuous with respect to the metric $d_{\cP(G)}(\cdot,\cdot)$, and thus is measurable with respect to the Borel $\sigma$-algebra $\cB(d_{\cP(G)})$. Now, as $\GeodSel: G \times G \to \Geod(G) \subset \cP(G)$ is Borel measurable (with respect to $\cB(d_{\cP(G)})$ on $\Geod(G)$), their composition $\bbS \circ \GeodSel: G \times G \to \cH$ is Borel measurable as well. 
\end{enumerate}
\end{proof}

In the rest of this section, we will record some important properties of the signature on $\cP(G)$. We also refer readers to \cite{LeeSig2020} for a comprehensive introduction of signature of paths in Lie groups.

\begin{proposition}\label{prop:reparametrization}
    The map $\mathbb S$ is unchanged under a reparametrization of $\gamma$.
\end{proposition}

\begin{proof}
    Suppose $\tilde \gamma :[a,b]\rightarrow G$ is a reparametrization of $\gamma: [0,1]\rightarrow G$, that is, there is an increasing $C^1$-function $t=\kappa(s)$ defined on $[a,b]$ such that $\kappa(a)=0$, $\kappa(b)=1$ and $\tilde \gamma=\gamma\circ \kappa$. Thus under change of variable $v_i=\kappa(u_i)$ we have
    \begin{align*}
    S_{i_1\cdots i_k}(\tilde \gamma)&=\int_{a<u_1<\cdots<u_k<b}\omega_{i_1}(\tilde \gamma'(u_1))\cdots\omega_{i_k}(\tilde \gamma'(u_k))du_1\cdots du_k\\
    &=\int_{0<v_1<\cdots<v_k<1}\omega_{i_1}(\gamma'(v_1) \kappa'(u_1))\cdots\omega_{i_k}(\gamma'(v_k)\kappa'(u_k))\frac{dv_1\cdots dv_k}{\kappa'(u_1)\cdots \kappa'(u_k)}\\
    &=\int_{0<v_1<\cdots<v_k<1}\omega_{i_1}(\gamma'(v_1) \cdots\omega_{i_k}(\gamma'(v_k)) dv_1\cdots dv_k\\
    &=S_{i_1\cdots i_k}(\gamma).
    \end{align*}
   \end{proof}

\begin{proposition}\label{prop:left-invariance}
    The signature is left invariant under the $G$-action, that is, for any $\gamma\in \mathcal P(G)$ and $g\in G$, we have
    \[\mathbb S(\gamma)=\mathbb S(g\gamma),\]
    where $(g\gamma)(t):= L_g(\gamma(t)) = g\gamma(t)\in \mathcal P(G)$.
\end{proposition}

\begin{proof}
    Pick any frame $\{e_1,\dots, e_n\}$ in $\mathfrak g$. It suffices to show the equality for each component.
    \begin{align*}
        S_{i_1\cdots i_k}(g\gamma)=\int_{0<u_1<...<u_k<1}\omega_{i_1}((g\gamma)'(u_1))\cdots\omega_{i_k}((g\gamma)'(u_k))du_1\cdots du_k.
    \end{align*}
    but since $(g\gamma)'(t) = g_*\gamma'(t) :=d(L_g)_{\gamma(t)}\gamma'(t)$ and each $\omega_i$ is left $G$-invariant, we have $\omega_i(g_*\gamma'(t))=\omega_i(\gamma'(t))$ for any $g\in G$ and $i\in \{1,\dots,n\}$. Thus it follows by definition that $S_{i_1\cdots i_k}(g\gamma)=S_{i_1\cdots i_k}(\gamma)$.
\end{proof}

We end this section with the following result of Chen \cite{Chen58}, although for the purpose of this paper we will not use it. Recall that for two (piecewise) smooth $G$-valued paths $\gamma_1$ and $\gamma_2$ in $\cP(G)$, their concatenation $\gamma = \gamma_1 * \gamma_2$ is a continuous and piecewise smooth $G$-valued path such that
$$
\gamma(t) = \begin{cases}
    \gamma_1(2t), \quad t \in [0,\frac{1}{2}];\\
    \gamma_1(1)\gamma_2(0)^{-1}\gamma_2(2t-1), \quad t \in [\frac{1}{2}, 1]. 
\end{cases}
$$
By Propositions \ref{prop:reparametrization} and \ref{prop:left-invariance}, we have $\bbS(\gamma)_{0,\frac{1}{2}} = \bbS(\gamma_1)$ and $\bbS(\gamma)_{\frac{1}{2},1} = \bbS(\gamma_2)$, where $\bbS(\gamma)_{s,t}$ denotes the signature of the segment $\gamma$ restricted on the time sub-interval $[s,t] \subset [0,1]$.

\begin{theorem}(Chen's identity \cite{Chen58})
If $\gamma$ is a concatenation of $\gamma_1 \in \cP(G)$ and $\gamma_2 \in \cP(G)$, then
\[\mathbb S(\gamma)=\bbS(\gamma)_{0,\frac{1}{2}} \otimes \bbS(\gamma)_{\frac{1}{2},1} = \mathbb S(\gamma_1)\otimes \mathbb S(\gamma_2).\]
\end{theorem}

As before, let $\widetilde \cP(G) = \cP(G)/\sim_{tree}$ be the quotient space of $\cP(G)$ modulo the tree-like equivalence relation, then thanks to the above Chen's identity and the injectivity of $\bbS$ on $\widetilde \cP(G)$ (see Remark \ref{remark: important properties of signature of Lie group valued paths}), the signature mapping $\bbS: \widetilde \cP(G) \to T((\fg))$ is a faithful group homomorphism, where $\widetilde \cP(G)$ is equipped with the concatenation as the group operation and $T((\fg))$ is endowed with the tensor product as the group operation. The group $\widetilde \cP(G)$ equipped with the concatenation is often called the reduced path group in the signature literature.


\section{The average signature}\label{sec:average-signature}
For now we have defined for any compact connected Lie group $G$ a natural coordinate-free notion of signature map $\mathbb S$. This leads to the following natural question: 

\begin{question}\label{ques:A(G)}
What is the average signature of geodesic paths in $G$?
\end{question}

If we take a closer look at the question, we quickly realize that given any two elements $g, h\in G$, there might exist infinitely many geodesic paths connecting $g, h$, even if we restrict our attention to those length minimizing ones, the geodesic path might not be unique. However, the non-uniqueness can be characterized by the cut locus. We give a brief summary of this notion and refer to \cite[Chapter III]{Chavel06} for more details.

Given a complete Riemannian manifold $M$, any point $p\in M$, and any geodesic ray $\gamma$ emanating from $p$, we say that a point $q$ on $\gamma$ is a cut point of $p$ along $\gamma$, if along $\gamma$ the geodesic path is not minimizing in length after $q$. The collection of all cut points of $p$ is called the cut locus of $p$. The following two properties on the cut locus will be crucial for us:
\begin{enumerate}
    \item[$\bullet$] If $q$ is not in the cut locus of $p$, then there exists a unique length minimizing geodesic segment connecting $p$ and $q$.
    \item[$\bullet$] For any $p\in M$, the set of cut locus of $p$ is of measure zero with respect to the Riemannian volume measure induced by the given Riemannian metric.
\end{enumerate}

Now we return to the discussion of Question \ref{ques:A(G)}. Given any two elements $g, h\in G$, we denote a shortest $C^1$-smooth (constant speed) geodesic path $\gamma_{g,h} \in \Gamma_{g,h} \subset \Geod(G)$ in $G$ connecting $g$ and $h$. This geodesic path might not be unique when $h$ lies in the cut locus of $g$. However, for each fixed $g\in G$, the cut locus of $g$ is of $\mu$-measure zero where $\mu$ is the volume form induced by the given bi-invariant metric $\langle \cdot,\cdot \rangle_G$ on $G$, so there is a full measure set of $h\in G$ that connects to $g$ via a unique shortest geodesic path, i.e., $\Gamma_{g,h} = \{\gamma_{g,h}\}$ is a singleton. For this reason, we may write the geodesic selector $\GeodSel$ defined in Proposition \ref{prop: measurable selector} as $\GeodSel(g,h) = \gamma_{g,h}$. By Proposition \ref{prop: measurable selector}, the mapping $(g,h) \in G \times G \mapsto \bbS(\gamma_{g,h}) \in \cH$ is Borel measurable;  by Proposition \ref{prop: boundedness of signature}, the (Hilbert) norm of $\bbS(\gamma_{g,h})$ is uniformly bounded by a constant for all $(g,h) \in G \times G$. Therefore, given the volume measure $\mu$ on $G$, 
the following $\cH$-valued Bochner integral
\[\mathbb A(G)=\frac{1}{\mu(G)\times \mu(G)}\int_{G\times G} \mathbb S(\gamma_{g,h}) d\mu(g)d\mu(h)\]
is well defined.
Geometrically, it represents the average value of signature of a generic geodesic path on $G$.

Using the left invariance of the signature (Proposition \ref{prop:left-invariance}) and the left invariance of the Haar measure, we can simplify the above expression as follows:
\begin{align*}
    \mathbb A(G)&=\frac{1}{\mu(G)\times \mu(G)}\int_{G\times G} \mathbb S(\gamma_{g,h}) d\mu(g)d\mu(h)\\
    &=\frac{1}{\mu(G)\times \mu(G)}\int_{G\times G} \mathbb S(\gamma_{e,g^{-1}h}) d\mu(g)d\mu(h)\\
    &=\frac{1}{\mu(G)}\int_{G}\left(\frac{1}{\mu(G)}\int_G \mathbb S(\gamma_{e,g^{-1}h}) d\mu(h)\right)d\mu(g)\\
    &=\frac{1}{\mu(G)}\int_{G}\left(\frac{1}{\mu(G)}\int_G \mathbb S(\gamma_{e,g^{-1}h}) d\mu(g^{-1}h)\right)d\mu(g)\\
    &=\frac{1}{\mu(G)}\int_G \mathbb S(\gamma_g)d\mu(g),
\end{align*}
where $\gamma_g := \gamma_{e,g}$ is the shortest geodesic connecting $e$ to $g$ in $G$. (Note that $\gamma_g$ is well defined for almost every point $g\in G$ except possibly at the measure zero cut locus.)
This leads to the following definition.

\begin{definition}\label{def:average-signature} Let $G$ be a connected, compact Lie group with a Haar measure $\mu$. The average signature of $G$ is given by
\begin{equation}\label{def:E(G)}
\mathbb A(G):=\frac{1}{\mu(G)}\int_G \mathbb S(\gamma_g)d\mu(g),
\end{equation}
For convenience, we also write $\mathbb A(G)$ in terms of the formal power series
\[\mathbb A(G)=\sum_{k=0}^\infty A_k(G),\]
where 
\[A_k(G)=\frac{1}{\mu(G)}\int_G S_k(\gamma_g)d\mu(g).\]
When the context is clear, we abbreviate $A_k(G)$ with $A_k$.
\end{definition}

 Since the Lie group $G$ is compact, the volume form $\mu$ induced by the given bi-invariant metric $\langle \cdot, \cdot \rangle_G$ is a finite Haar measure on $G$ (see e.g. \cite[Theorem 1.46]{CompactLieGroups}), which guarantees that $\mathbb A(G) = \frac{1}{\mu(G)}\int_G \mathbb S(\gamma_g)d\mu(g)$ as a Bochner integral against a finite measure is well-defined (recall that by  Proposition \ref{prop: boundedness of signature}, $\|\mathbb S(\gamma_g)\|_{\cH}$ is uniformly bounded for all geodesics $\gamma_g$).
Moreover, it is well known (see e.g. \cite[Lemma 1.44]{CompactLieGroups}) that $\mu$ is the unique Haar measure on $G$ up to a nonzero scalar, that is, if $\tilde \mu$ is another (nonzero) Haar measure on $G$, then there exists a nonzero constant $c$ such that $\mu = c \tilde \mu$. In particular, if $\widetilde{\langle \cdot, \cdot \rangle_G}$ is another 
bi-invariant metric on $G$ and $\tilde \mu$ is the attached volume form (which is a Haar measure on $G$), then we have $\mu = c \tilde \mu$ for some nonzero constant $c$.

On the other hand, given any bi-invariant metric on a connected compact Lie group $G$, the Riemannian exponential map on $G$ with respect to the bi-invariant metric coincides with the surjective exponential map $\exp:\fg \to G$ of the Lie group $G$ (\cite[Theorem 2.27]{LieGroupsandGeometricAspects}). In particular, for any $g\in G$ outside a $\mu$-measure zero cut locus of $e$, there exists a unique $v(g)\in \mathfrak g$ which is independent of the choice of the bi-invariant metric such that the length minimizing geodesic segment connecting $e$ to $g$ is exactly an interval of the one-parameter family $t\mapsto\exp(tv(g)), t\in[0,1]$. 
Moreover, the assignment $g\mapsto v(g)$ can be viewed as an ``almost'' inverse map of $\exp$, and it is only defined on some full $\mu$-measure set $G_0\subset G$. Note that for any $g \in G_0$, since $\gamma_g(t)=\GeodSel(e,g)(t) = \exp(tv(g))$, it implies that the $\fg$-valued path $\Phi(\gamma_g)$ fulfills that $\Phi(\gamma_g)^\prime(t) = (L_{\gamma_g(t)^{-1}})_*(\gamma_g^\prime(t)) = (L_{\gamma_g(t)^{-1}})_* (L_{\gamma_g(t)})_*v(g) = v(g)$ for all $t \in [0,1]$, i.e., $\Phi(\gamma_g)(t) = t v(g)$ is a straight line segment from $0 \in \fg$ to $v(g)$. Hence, we have
$$
v(g) = \Phi(\GeodSel(e,g))^\prime(t)
$$
for all $t \in [0,1]$, so that by the measurability of the function $\GeodSel$ (see Proposition \ref{prop: measurable selector}), one can easily check that the mapping $v(\cdot): G_0 \to \fg$ is a Borel measurable function on $G_0$. By setting $v(h) = 0 \in \fg$ for all $h \in G\setminus G_0$, we can and will view $v(\cdot): G \to \fg$ as a measurable $\fg$-valued function defined  on the whole Lie group $G$.

In fact, using this map $v:G_0\rightarrow \mathfrak g$, we obtain the following observations for the average signature:  since $\gamma_g(t) = tv(g)$ for $g \in G_0$ and $t \in [0,1]$, using Proposition \ref{prop: two signatures are same}, it holds that 
\begin{equation}\label{eq:sig-simple}
    \bbS(\gamma_g) = \bS(\Phi(\gamma_g)) = \exp_{\otimes}(v(g)) = \sum_{k=0}^\infty \frac{v(g)^{\otimes k}}{k!},
\end{equation}
where $\exp_{\otimes}: \fg \to T((\fg))$ denotes the tensorial exponential mapping. Thus, by writing $\bar \mu = \frac{\mu}{\mu(G)}$ as the unique Haar probability measure on $G$, we have

\begin{equation}\label{eq:ave-sig-simple}
    \mathbb A(G)=\frac{1}{\mu(G)}\int_G \mathbb \exp_{\otimes}(v(g))d\mu(g) = \int_G \exp_{\otimes}(v(g))d\bar \mu(g),
\end{equation}
is nothing else but the moment generating function of the random variable $v(\cdot)$ on $G$ with respect to $\bar \mu$; or equivalently speaking, the expected signature of the pushforward probability measure $(g \mapsto \GeodSel(e,g))_\sharp\bar \mu$ on $\cP(G)$: $\mathbb A(G)= \mathbb E_{\bar \mu}[\bbS (\GeodSel(e,\cdot) ) ]$.
Similarly, for any integer $k \ge 1$,
\begin{equation}\label{eq: expression of Ak}
     A_k(G)=\frac{1}{\mu(G)}\int_G \frac{v(g)^{\otimes k}}{k!}d\mu(g) = \int_G \frac{v(g)^{\otimes k}}{k!}d\bar \mu(g)
\end{equation}
is the $k$-th moment of the $\fg$-valued random variable $v(\cdot)$ against $\bar \mu$. We refer readers to \cite{Chevyrev2016chf} and \cite{Chevyrev2022signaturekernel} for a comprehensive exploration to the theory of expected signature.

 Combining all above observations, we can conclude that $\mathbb A(G)$ does not depend on the choice of the bi-invariant metric, and the notion of average signature $\mathbb A(G)$ is a natural invariant associated to any compact connected Lie group $G$. The rest of the paper is devoted to the study of this invariant.

\begin{remark}\label{remark: average signature is symmetric}
\begin{enumerate}
    \item In view of \eqref{eq: expression of Ak}, it is clear that for any integer $k \ge 1$, the $k$-th component of the average signature $\mathbb A(G)$, $A_k(G) = \int_G \frac{v(g)^{\otimes k}}{k!}d\bar \mu(g)$ actually takes values in the symmetric tensor space $\mathcal S_k(\fg) \subset (\fg)^{\otimes k}$, where $\mathcal S_k(\fg) = \{\mathbf{v} \in (\fg)^{\otimes k}: \forall \sigma \in S_k, \mathbf{v}^\sigma = \mathbf{v} \}$, $S_k$ is the permutation group of $\{1,\ldots,k\}$ and $\mathbf{v} \mapsto \mathbf{v}^{\sigma}$ is a the linear mapping such that $\mathbf{v}^{\sigma} = v_{\sigma(1)} \otimes \ldots \otimes v_{\sigma(k)}$ if $\mathbf{v} = v_1 \otimes \ldots \otimes v_k$ with $v_i \in \fg$ for $i=1,\ldots,k$, see \cite{SmoothmanifoldsLee} for more details.  This symmetric property of average signature $\mathbb A(G)$ will play an important rule in the recovery of the geometric information, see the next section.
    \item As we shall see later, the fact that the geodesics on the compact Lie group are essentially straight lines in its Lie algebra is indeed crucial for us to use the average signature to recover certain geometric properties of the underlying Lie group, and it explains why the compact Lie group setting is tractable and
    distinguishes it from general Riemannian manifolds. For example, the (average) signature of geodesics on general Riemannian manifolds may not be a symmetric tensor anymore, so that the notion of the trace spectrum (which will be introduced in the next section) cannot be easily generalized to the average signature on general Riemannian manifolds.
\end{enumerate}

\end{remark}

\begin{example}\label{ex:circle}
$G=S^1$.\\

We identify $G$ with the standard unit circle, parametrized by $\theta\in (-\pi,\pi]$ where $\theta=0$ corresponds to the identity element. Thus $\mathfrak g$ is a one dimensional Lie algebra generated by $e_1=(0,1)$, that is dual to the one-form $d\theta$. The normalized Haar measure is given by $d\mu=\frac{1}{2\pi}d\theta$. For any $\theta\in G$ such that $\theta\neq \pi$, there is a unique geodesic $\gamma_\theta(t)$ connecting the identity $e$ to $\theta$ in $G$. Now we can compute $\mathbb A(G)$ as follows.

We write $\mathbb A(G)=(A_0, A_1,\dots, A_n,\dots)$ where
\[A_n=\frac{1}{2\pi}\int_{-\pi}^\pi S_n(\gamma_\theta)d\theta.\]
Since 
\begin{align*}
    S_n(\gamma_\theta)&=\begin{cases}
        \left(\int_{0<u_1<\dots<u_n<\theta} du_1\dots du_n \right)e_1^{\otimes n}& \theta\in[0,\pi)\\
        \left(\int_{0<u_1<\dots<u_n<|\theta|} (-1)^n du_1\dots du_n\right)e_1^{\otimes n}& \theta\in(-\pi,0)
    \end{cases}\\
    &=\frac{\theta^n}{n!}e_1^{\otimes n},\quad \theta\in (-\pi, \pi)
\end{align*}
We have
\begin{align*}
    A_n&=\frac{1}{2\pi}\left(\int_{-\pi}^\pi \frac{\theta^n}{n!}d\theta \right)e_1^{\otimes n}\\
    &=\begin{cases}
      \frac{\pi^n}{(n+1)!}e_1^{\otimes n}  & n \textrm{ is even}\\
     0   & n \textrm{ is odd}
    \end{cases}  
\end{align*}
Thus we can simplify the formal power series as
\begin{align*}
    \mathbb A(S^1)=\sum_{k=0}^\infty \frac{(\pi e_1)^{\otimes 2k}}{(2k+1)!}=\frac{\sinh (\pi e_1)}{\pi e_1}.
\end{align*}

\end{example}

\begin{example}\label{ex:3sphere}
    $G=\SU(2)$.\\
    
For the purpose of computation, it is convenient to pull the integral \eqref{def:E(G)} back to the Lie algebra $\mathfrak g\cong \su(2)$ via the exponential map. We choose a standard basis $\{e_1, e_2, e_3\}$ on $\su(2)$ such that
\[e_1=\left(\begin{matrix}
i & 0 \\
0 & -i 
\end{matrix}\right),
e_2=\left(\begin{matrix}
0 & 1 \\
-1 & 0 
\end{matrix}\right),
e_3=\left(\begin{matrix}
0 & i \\
i & 0 
\end{matrix}\right),
\]
with the Lie bracket relations given by $[e_1,e_2]=2e_3, [e_1,e_3]=-2e_2, [e_2,e_3]=2e_1$. We identify $\mathfrak g$ with $\mathbb R^3$ under such basis. Note that the exponential map $\exp: \su(2)\cong \mathbb R^3 \rightarrow G$ is a homeomorphism when restricted to the open ball $B(\pi):=\{(\lambda_1,\lambda_2,\lambda_3):\lambda_1^2+\lambda_2^2+\lambda_3^2<\pi^2\}$, whose image is exactly $G\setminus \{-I\}$ in $G$. Now pulling back the integral on \eqref{def:E(G)} gives
\begin{align*}
\mathbb A(G)&=\frac{1}{\mu(G)}\int_G \mathbb S(\gamma_g)d\mu(g)\\
&=\frac{1}{\mu(G)}\int_{B(\pi)} \mathbb S(\lambda_1,\lambda_2,\lambda_3) \cdot |\Jac_{(\lambda_1,\lambda_2,\lambda_3)}(\exp)|d\lambda_1d\lambda_2d\lambda_3,
\end{align*}
where $\mathbb S(\lambda_1,\lambda_2,\lambda_3)=\mathbb S(\gamma_g)$ for which $g=\exp(\lambda_1e_1+\lambda_2e_2+\lambda_3e_3)$.

We now compute the Jacobian of the exponential map. For any $C^1$-curve $A(t)\in\mathfrak g$ such that $A(0)=A$, according to Duhamel's formula \cite[Section 1.2, Theorem 5]{Rossmann02},
\[\frac{d}{dt}\Bigg|_{t=0}\exp(A(0))^{-1}\exp(A(t))=\frac{1-e^{-\ad_A}}{\ad_A}A'(0),\]
where $\ad_A: \fg \to \fg, \ad_A(B)=[A,B]$ (where $[A,B] = AB - BA$ is the natural Lie bracket on matrix Lie algebras) is the adjoint action on $\mathfrak g$, and
\[\frac{1-e^{-\ad_A}}{\ad_A}=\sum_{k=0}^\infty \frac{(-1)^k}{(k+1)!}(\ad_A)^k\]
is the corresponding formal power series of the operator $\ad_A$ on $\mathfrak g$. This shows that 
\[\Jac_A(\exp)=\det \left(\frac{1-e^{-\ad_A}}{\ad_A}\right).\]
When $A=(\lambda_1,\lambda_2,\lambda_3)$, we can compute explicitly under the basis $\{e_1,e_2,e_3\}$ that
\[\ad_A=2\left(\begin{matrix}
    0 & -\lambda_3 & \lambda_2\\
    \lambda_3 & 0 & -\lambda_1\\
    -\lambda_2 & \lambda_1 & 0
\end{matrix}\right)\]
Since $\ad_A$ has three eigenvalues $0, \pm 2\mu$ where $\mu=\sqrt{\lambda_1^2+\lambda_2^2+\lambda_3^2}\cdot i$, we see that 
\[\Jac_A(\exp)=1\cdot \frac{1-e^{-2\mu}}{2\mu}\cdot \frac{1-e^{2\mu}}{-2\mu}=\frac{\sin^2 r}{r^2},\]
where $r=\sqrt{\lambda_1^2+\lambda_2^2+\lambda_3^2}$.

Next we compute $\mathbb S(\lambda_1,\lambda_2,\lambda_3)$. For any $g\in G\setminus \{-I\}$, the unique geodesic connecting $I$ to $g$ in $G$ is given by $\gamma_g(t)=\exp(t\cdot v(g))$ where $v(g)$ is the unit vector in the direction of $\exp^{-1}(g)$ in $\mathfrak g$.
Adopting the same notation that $g=\exp(\lambda_1e_1+\lambda_2e_2+\lambda_3e_3)$ and $r=\sqrt{\lambda_1^2+\lambda_2^2+\lambda_3^2}$, we have
$\gamma_g'(t)=v(g)=\frac{1}{r}(\lambda_1e_1+\lambda_2e_2+\lambda_3e_3)$. Thus, it follows from \eqref{eq:sig-simple} that
\[\mathbb S(\lambda_1,\lambda_2,\lambda_3)=\sum_{n=0}^\infty\frac{(\lambda_1e_1+\lambda_2e_2+\lambda_3e_3)^{\otimes n}}{n!}=e^{(\lambda_1e_1+\lambda_2e_2+\lambda_3e_3)}.\]

Finally we compute $\mathbb A(G)$ using the spherical coordinate. Combining the above we obtain,
\begin{align*}
    \mathbb A(G)&=\frac{1}{\mu(G)}\int_{B(\pi)} \mathbb S(\lambda_1,\lambda_2,\lambda_3) \cdot |\Jac_{(\lambda_1,\lambda_2,\lambda_3)}(\exp)|d\lambda_1d\lambda_2d\lambda_3\\
    &=\frac{1}{\vol(S^3)}\int_0^\pi r^2 dr\int_{S^2}\sum_{n=0}^\infty \frac{(r\cdot v)^{\otimes n}}{n!}\cdot \frac{\sin^2 r}{ r^2} dS^2(v)\\
    &=\frac{1}{2\pi^2}\sum_{n=0}^\infty \left(I_n\cdot \int_{S^2}\frac{v^{\otimes n}}{n!}dS^2(v)\right),
\end{align*}
where $I_n=\int_0^\pi (r^n\sin^2 r) dr$ satisfies the recursive relations $I_0=\frac{\pi}{2}$ and $I_n=\frac{n(n-1)}{4}I_{n-2}+\frac{\pi^{n+1}}{2(n+1)}$. We also note that since the spherical measure $dS^2$ is flip ($v\mapsto -v$) invariant, the odd degree terms in $\mathbb A(G)$ all vanish.
\end{example}

\begin{example}\label{ex:torus}
    $G=T^2=S^1\times S^1$.\\
    
This example illustrates the case where $G$ is not simple, and consequently the bi-invariant metric on $G$ is not unique upto scale. We give the natural parameterization $(\theta_1,\theta_2)$ on $G$ where $\theta_i\in(-\pi,\pi]$ is the rotational angle on each factors. Similar to the above Example \ref{ex:circle}, we denote $e_i$ the tangent vector fields dual to $d\theta_i$, then the Lie algebra $\mathfrak g$ is a two dimensional abelian algebra spanned by $\{e_1, e_2\}$. Fix any bi-invariant metric on $G$, for each $\theta=(\theta_1,\theta_2)\in (-\pi,\pi)\times (-\pi,\pi)$ there is a unique length minimizing geodesic $\gamma_\theta$ on $G$ connecting $(0,0)$ to $(\theta_1,\theta_2)$ given by the line segment. Since the signature does not depend on the parameterization, we have
\begin{align*}
    S_n(\gamma_\theta)&=\frac{(\theta_1e_1+\theta_2e_2)^{\otimes n}}{n!}\\
    &=\sum_{\mathcal I=(i_1,\cdots,i_n)\in\{1,2\}^n}\frac{\theta_1^{k(\mathcal I)}\cdot \theta_2^{\ell(\mathcal I)}}{n!}e_{i_1}\otimes\cdots\otimes e_{i_n},
\end{align*}
where $k(\mathcal I), \ell(\mathcal I)$ denote the total number of $1$s and $2$s appearing in the multi-index $\mathcal I=(i_1,\cdots,i_n)$, respectively. By taking the average, we obtain that
\begin{align*}
    A_n&=\frac{1}{4\pi^2}\iint_{(-\pi,\pi)\times(-\pi,\pi)}S_n(\gamma_\theta)d\theta_1d\theta_2\\
    &=\frac{1}{4\pi^2 n!}\sum_{\mathcal I=(i_1,\cdots,i_n)\in\{1,2\}^n}\frac{\left(\left.\theta_1^{k(\mathcal I)+1}\right\vert_{-\pi}^\pi\right)\left(\left.\theta_2^{\ell(\mathcal I)+1}\right\vert_{-\pi}^\pi\right)}{(k(\mathcal I)+1)(\ell(\mathcal I)+1)}e_{i_1}\otimes\cdots\otimes e_{i_n}\\
    &=\frac{\pi^n}{n!}\sum_{\substack{\mathcal I=(i_1,\cdots,i_n)\in\{1,2\}^n\\ k(\mathcal I),\ell(\mathcal I) \textrm{ are even}}}\frac{1}{(k(\mathcal I)+1)(\ell(\mathcal I)+1)}e_{i_1}\otimes\cdots\otimes e_{i_n}.
\end{align*}
Note that $\mathbb A(T^2)=\sum_{n=1}^\infty A_n$ does not depend on the choice of bi-invariant metric on $G$.
\end{example}

To conclude this section, we show a crucial property of the average signature of $G$ which will be used heavily later on, namely as a series in $T((\fg)) = \prod_{k=0}^\infty \fg^{\otimes k}$, all of its odd degree components vanish.
\begin{proposition}\label{prop:even-degree}
For any connected, compact Lie group $G$, the average signature $\mathbb A(G)$ consists of only even degree terms, that is, $A_k=0$ whenever $k$ is odd.
\end{proposition}

\begin{proof}
First, since the integration of functions on compact Lie groups with respect to Haar measures is invariant under the inversion $g \in G \mapsto g^{-1} \in G$, see e.g.  \cite[Theorem 1.46]{CompactLieGroups}, we have
$$
\mathbb{A}(G) = \frac{1}{\mu(G)}\int_G \bbS(\gamma_{g})d\mu(g) = \frac{1}{\mu(G)}\int_G \bbS(\gamma_{g^{-1}})d\mu(g),
$$
which implies that for every $k\ge 0$,
\begin{align*}
    A_k = \frac{1}{2\mu(G)}\int_G \left(S_k(\gamma_g)+S_k(\gamma_{g^{-1}})\right)d\mu(g).
\end{align*} 
Now we compute the $k$-th components of signatures $S_k(\gamma_g)$ and $S_k(\gamma_{g^{-1}})$ for any $k \ge 0$ and any $g \in G$. Since $G$ is a connected compact Lie group equipped with a bi-invariant metric, by \cite[Theorem 2.27]{LieGroupsandGeometricAspects} we know that the exponential map $\exp: \fg \to G$ is surjective, see also the comments between Definition \ref{def:average-signature} and Example \ref{ex:circle}. 
Furthermore, since the inverse map $g\mapsto g^{-1}$ is an isometry under any bi-invariant metric, the geodesic segment $\exp(tv(g)),t\in[0,1]$ connecting $e$ to $g$ is length minimizing if and only if $\exp(-tv(g)),t\in[0,1]$ is a length minimizing geodesic connecting $e$ to $g^{-1}$. This shows that $\bbS(\gamma_{g^{-1}}) = \bS(\Phi(\gamma_{g^{-1}})) = \exp_{\otimes}(-v(g)) = \sum_{k=0}^\infty \frac{(-v(g))^{\otimes k}}{k !}$. So, to summarize, we have 
\[S_k(\gamma_g)=\frac{v(g)^{\otimes k}}{k!},\textrm{ and } S_k(\gamma_{g^{-1}})=\frac{(-v(g))^{\otimes k}}{k!}.\]
It follows immediately that  $S_k(\gamma_g)+S_k(\gamma_{g^{-1}})=0$ whenever $k$ is odd, and thus $A_k = \frac{1}{2\mu(G)}\int_G \left(S_k(\gamma_g)+S_k(\gamma_{g^{-1}})\right)d\mu(g)=0$ for all odd $k$, which completes the proof of the proposition.
\end{proof}

\section{Trace spectrum of the average signature}\label{sec:tr}
As we see in Example \ref{ex:circle}, \ref{ex:3sphere}, \ref{ex:torus}, the average signature should naturally contain information of the underlying Lie group, but it is very hard to compute in reality, partly because it values in a huge tensor algebra of the corresponding Lie algebra. On the other hand, it is  invariant under the choice of the underlying bi-invariant Riemannian metric on $G$; hence the average signature itself does not recover any geometric information of the Lie group. For this reason, it is important for us to combine the average signature with  an operation induced by the bi-invariant Riemannian metric put on $G$ in the task of recovering geometric features of $G$.  As we will see later, such an operation is given by taking the trace of average signature with respect to the given bi-invariant metric. By imposing such additional information provided by the trace operator, we can now extract certain scalar information by taking the tensor contractions, and eventually recover certain geometric invariants of the underlying Lie group.

Let $\mathfrak g$ be equipped with the the inner product which arises as 
the restriction of a fixed bi-invariant metric $\langle \cdot, \cdot \rangle_G$ to $\fg \cong T_eG$, then the dual space $\mathfrak g^*$ is naturally identified with $\mathfrak g$. We now consider a tensor contraction linear functional on $\mathfrak g^{\otimes 2k}$  with respect to $\langle \cdot, \cdot \rangle_G$ for any $k \ge 1$, which can be viewed as a higher order version of matrix trace.

\begin{definition}
    Let $k\ge 1$ be an integer. For any $A\in \mathfrak g^{\otimes 2k}$, the trace is given by pairwise taking the contractions on each $(2i-1,2i)$ coordinates, for $1\leq i\leq k$. More precisely, if $\{e_1,\dots,e_{n}\}$ is an orthonormal basis of $\mathfrak g$ with respect to $\langle \cdot, \cdot \rangle_G$, and we write
    \[A=\sum_{1\leq i_1,\dots,i_{2k}\leq n} a_{i_1\cdots i_{2k}}e_{i_1}\otimes \cdots\otimes e_{i_{2k}},\]
    then
    \[\tr(A):=\sum_{1\leq i_1,\dots,i_{2k}\leq n}a_{i_1\cdots i_{2k}}\delta_{i_1i_2}\cdots\delta_{i_{2k-1}i_{2k}}.\]
One easily verifies that the definition is independent of the choice of the orthonormal basis on $\mathfrak g$ (but it does depend on the choice of the inner product/Riemannian metric on $\fg$). For simplicity, we extend the definition to tensors of odd rank by declaring the trace is always zero.
\end{definition}

It is clear that the trace is a linear operator, that is, for any $A, B\in \mathfrak g^{\otimes 2k}$, and any $\lambda\in \mathbb R$, we have
$\tr(A+B)=\tr(A)+\tr(B)$ and $\tr(\lambda A)=\lambda\tr(A)$. Moreover, we have the following simple relation between the trace on $\mathfrak g^{\otimes 2k}$ and the inner product on $\mathfrak g$.
\begin{lemma}\label{lem:tr}
    If $A=u_1\otimes v_1\otimes \cdots\otimes u_k\otimes v_k \in \mathfrak g^{\otimes 2k}$, where $u_i,v_i\in \mathfrak g$ for each $1\leq i\leq n$, then
    \[\tr (A) =\prod_{i=1}^k \langle u_i,v_i\rangle.\]
\end{lemma}
\begin{proof}
    Choose an orthonormal basis $\{e_1,\dots,e_n\}$ of $\mathfrak g$. Let $u_i=\sum_{j=1}^n a_{ij} e_{j}$ and $v_i=\sum_{j=1}^n b_{ij} e_j$. Then we compute $\tr A$ in the following,
    \begin{align*}
        &\tr\left[ \left(\sum_{j_1=1}^n a_{1j_1} e_{j_1}\right)\otimes \left(\sum_{l_1=1}^n b_{1l_1} e_{l_1}\right)\otimes\cdots\otimes \left(\sum_{j_k=1}^n a_{kj_k} e_{j_k}\right)\otimes \left(\sum_{l_k=1}^n b_{kl_k} e_{l_k}\right)\right]\\
      &=\tr\left[\sum_{1\leq j_1,\dots,j_k\leq n}\sum_{1\leq l_1,\dots,l_k\leq n} a_{1j_1}b_{1l_1}\cdots a_{kj_k}b_{kl_k} e_{j_1}\otimes e_{l_1}\otimes\cdots\otimes e_{j_k}\otimes e_{l_k}\right]\\
      &=\sum_{1\leq j_1,\dots,j_k\leq n}\sum_{1\leq l_1,\dots,l_k\leq n}\left(a_{1j_1}b_{1l_1}\cdots a_{kj_k}b_{kl_k}\right)\cdot\left(\delta_{j_1l_1}\cdots\delta_{j_kl_k}\right)\\
      &=\sum_{1\leq j_1,\dots,j_k\leq n}a_{1j_1}b_{1j_1}\cdots a_{kj_k}b_{kj_k}\\
      &=\prod_{i=1}^n \langle u_i,v_i\rangle.
    \end{align*}
\end{proof}

\begin{remark}\label{remark: on the trace}

 In view of the above lemma, we see that this trace operator is nothing but a pairwise contracting operator (or the evaluation map in some literature). More precisely, it is the following composition
\[\mathfrak g^{\otimes 2k}\to (\mathfrak g\otimes \mathfrak g^*)^{\otimes k}\to \mathbb R,\]
where the first map is the canonical linear isomorphism induced by $\mathfrak g\cong \mathfrak g^*$ (where $\fg^*$ is the dual space of $\fg$) that is determined by the inner product, and the second map is simply the product of the $k$-pairs of contractions. More precisely, 
\begin{itemize}
    \item When $k=1$, we treat each vector $A = u_1 \otimes v_1 \in \fg^{\otimes 2}$ as $A = u_1 \otimes v_1^* \in \fg \otimes \fg^*$ with $v_1^* \in \fg^*$ being the dual on $\fg$ given by $v_1^*(w) = \langle v_1, w \rangle_G$. In many literature on differential geometry, see e.g. \cite{RiemannianManifoldsLee}, the vector $u_1 \otimes v_1 \in \fg^{\otimes 2}$ is called a $(2,0)$-tensor and $u_1 \otimes v_1^* \in \fg \otimes \fg^*$ is called a $(1,1)$-tensor. This $(1,1)$-tensor $u_1 \otimes v_1^*$ can be viewed as an endomorphsim on $\fg$ through the rule $u_1 \otimes v_1^*(w) = v_1^*(w)u_1$. In this case the term $\tr(A) = \tr(u_1 \otimes v_1^*)$ is exactly the classical trace of this endomorphsim $w \mapsto v_1^*(w)u_1$.
    \item For a general integer $k \ge 1$, we treat $A = u_1\otimes v_1\otimes \cdots\otimes u_k\otimes v_k \in \mathfrak g^{\otimes 2k}$ as $A = u_1 \otimes v_1^* \otimes u_2 \otimes v_2^* \otimes \ldots \otimes u_k \otimes v_k^* \in (\mathfrak g\otimes \mathfrak g^*)^{\otimes k}$, and the vector $A = u_1 \otimes v_1^* \otimes u_2 \otimes v_2^* \otimes \ldots \otimes u_k \otimes v_k^*$ is called a $(k,k)$-tensor. By taking the trace (also called the contraction) for the first pair $u_1$ and $v_1^*$ while keeping $u_2,v_2^*,\ldots,u_k,v_k^*$ unchanged, we can obtain a $(k-1,k-1)$-tensor which is denoted by $\tr_{1}(A)$, see \cite[pp. 13, pp.28--29]{RiemannianManifoldsLee} for the definition of such trace/contraction operation for general $(k,l)$-tensors. Then, by taking the trace/contraction to $\tr_{1}(A)$ for the first pair $u_2$ and $v_2^*$ in $\tr_{1}(A)$, we can get a $(k-2,k-2)$-tensor denoted by $\tr_2 \circ \tr_1 (A)$. Continuing this procedure, we finally arrive at
    $$
    \tr(A) = \tr_k \circ \tr_{k-1} \circ \ldots \circ \tr_1 (A), 
    $$
    which is a real number which is identified with a $(0,0)$-tensor.  
\end{itemize}
\end{remark}

\begin{remark}\label{remark: trace is well-defined}

It may seem ``unnatural'' from the definition since we made the choice of the linear isomorphism on the first map by dualizing all the even terms. However, the signature of the geodesic paths that we are considering actually lies in the subspace of the symmetric $2k$-tensor, and so does its average, see Remark \ref{remark: average signature is symmetric}. For symmetric $2k$-tensors, the choice of pairing and contraction will not change the resulting value. For instance, for a symmetric tensor $A = \sum_{m} u^{(m)}_1\otimes v^{(m)}_1\otimes u^{(m)}_2\otimes v^{(m)}_2 \in \mathfrak g^{\otimes 4}$ we can also view it as $A = \sum_{m} u^{(m)}_1\otimes v^{(m)}_1\otimes (u^{(m)}_2)^*\otimes (v^{(m)}_2)^* \in \fg^{\otimes 2} \otimes (\fg^*)^{\otimes 2}$ and define $\tr^\prime(A)$ by pairing $u^{(m)}_1$ with $(u^{(m)}_2)^*$ and pairing $v^{(m)}_1$ with $(v^{(m)}_2)^*$. Then one can easily check that $\tr(A) = \tr^\prime(A)$, where $\tr(A)$ is defined as in Definition \ref{def:tr} by pairing $u^{(m)}_1$ with $(v^{(m)}_1)^*$ and pairing $u^{(m)}_2$ with $(v^{(m)}_2)^*$.
In this sense, the trace we define here is indeed a ``natural'' invariant associated to the average signature.

\end{remark}

Since $\mathbb A(G)$ only supports on even degree tensors (see Proposition \ref{prop:even-degree}), we can define its trace by only focusing on those terms with even degree.

\begin{definition}\label{def:tr}
 Let $G$ be a compact Lie group with a choice of a bi-invariant metric. The trace (spectrum) of the average signature is given by
    \[\tr(\mathbb A(G)) := (\tr (A_{k}))_{k=0}^\infty =(1,0,\tr A_2,0,\tr A_4,\cdots)\in \mathbb R^{\infty}.\]
\end{definition}

As we will see later, it is more convenient for us to work with the following rescaled trace spectrum.

\begin{definition}\label{def:rtr}
 We denote the rescaled trace (spectrum) of the average signature by  
 \[\rtr(\mathbb A(G))= (\rtr(A_{k}))_{k=0}^\infty =(1,0,\rtr(A_2),0,\rtr(A_4),\cdots)\in \mathbb R^{\infty},\]
 where
 \[\rtr (A_k)=(k!)\tr (A_k).\]
\end{definition}

 We emphasize again that the (rescaled) trace of the average signature, as opposed to the average signature itself, depends on the choice of the bi-invariant metric on $G$. (See Example \ref{ex:trace} below.) We also note that, for any simple compact Lie group $G$, there is a unique bi-invariant metric on $G$ upto a scalar multiple.

\begin{example}\label{ex:trace} Following from the computations in Example \ref{ex:circle}, \ref{ex:3sphere} and \ref{ex:torus}, we have

\begin{enumerate}
    \item[$\bullet$] $G=S^1$ with the bi-invariant metric $\langle e_1,e_1\rangle=\lambda^2$ for $\lambda>0$:
\[\tr(A_{2k}) =\frac{\lambda^{2k}\pi^{2k}}{(2k+1)!},\quad \rtr (A_{2k})=\frac{\lambda^{2k}\pi^{2k}}{2k+1}.\]
\item[$\bullet$] $G=\SU(2)$ with the bi-invariant metric $\langle e_i,e_j\rangle:=\frac{\lambda^2}{2}\tr(e_i\cdot e_j^*)$ for $\lambda>0$:
\[\tr(A_{2k}) =\frac{2\lambda^{2k}\cdot I_{2k}}{\pi(2k)!},\quad \rtr (A_{2k})=\frac{2\lambda^{2k}\cdot I_{2k}}{\pi}.\]
where $I_{2k}=\int_0^\pi (r^{2k}\sin^2 r) dr$ satisfies the recursive relations $I_0=\frac{\pi}{2}$ and $I_{2k}=\frac{k(2k-1)}{2}I_{2k-2}+\frac{\pi^{2k+1}}{2(2k+1)}$.
\item[$\bullet$] $G=S^1\times S^1$ with the family of bi-invariant metrics $\langle e_1,e_1\rangle=\lambda_1^2$, $\langle e_1,e_2\rangle=0$, and $\langle e_2,e_2\rangle=\lambda_2^2$, where $\lambda_1,\lambda_2>0$ are two positive real numbers:
\begin{align*}
    \tr(A_{2k})&=\frac{\pi^{2k}}{(2k)!}\sum_{\substack{\mathcal I=(i_1,\cdots,i_{2k})\in\{1,2\}^{2k}\\ k(\mathcal I),\ell(\mathcal I) \textrm{ are even}}}\frac{1}{(k(\mathcal I)+1)(\ell(\mathcal I)+1)}\tr(e_{i_1}\otimes\cdots\otimes e_{i_{2k}}).\\
    &=\frac{\pi^{2k}}{(2k)!}\sum_{\substack{i,j\geq 0\\i+j=k}}\sum_{\substack{\mathcal I=(i_1,i_1,\cdots,i_{k},i_k)\in\{1,2\}^{2k}\\k(\mathcal I)=2i,\;\ell(\mathcal I)=2j}}\frac{\lambda_1^{2i}\lambda_2^{2j}}{(2i+1)(2j+1)}\\
    &=\frac{\pi^{2k}}{(2k)!}\sum_{\substack{i,j\geq 0\\i+j=k}} {k\choose{i}}\frac{\lambda_1^{2i}\lambda_2^{2j}}{(2i+1)(2j+1)},
\end{align*}
where the second equality uses the fact that the trace is nonzero only when the multi-index is of the form $\mathcal I=(i_1,i_1,\cdots,i_{k},i_k)$, and the last equality simply counts the total number of this form with exactly $i$-pairs of $1s$ and $j$-pairs of $2s$ appearing in $\mathcal I$. From the expression, we see explicitly that the trace depends on the choice of bi-invariant metric on $G$.
\

\end{enumerate}   
\end{example}

We are interested to know whether/how the (rescaled) trace of the average signature is able to recover the geometry of $G$, optimally in an explicit way. Now we fix a bi-invariant metric on $G$. In particular, $G$ is now a Riemannian manifold. Moreover, recall that $\mu$ denotes the volume form on $G$ induced by this bi-invariant metric, which is a Haar measure. As before, for the Riemannian volume form $\mu$ on $G$ induced by the given bi-invariant Riemannian metric on $G$, we denote $\bar \mu = \frac{\mu}{\mu(G)}$ the unique Haar probability measure on $G$, also called the relative volume form. Note that the relative volume can be used to define a certain averaged amount of Ricci curvature which in turn provides a criterion for the (pre)compactness of Riemannian manifolds in the Gromov-Hausdorff topology, see \cite{RelativeVolume1997}.

\begin{theorem}\label{thm:Lk-norm}
    Let $f:G\rightarrow \mathbb R_{\geq 0}$ be the smooth function given by $f(g)=d(e,g)^2$ where $d$ is the Riemannian distance function on $G$. Then for any $k\in \mathbb N^+$, the $L^k$-norm of $f$ (with respect to the Haar probability measure $\bar \mu$) can be recovered from $\rtr (\mathbb A(G))$. More preciesly, we have
    \[||f||_{L^k}=\left(\rtr(A_{2k})\right)^{1/k}.\]
\end{theorem}

\begin{proof}
According to the definition and the linearity of trace, we have
\[\tr(A_{2k}) =\tr\left(\frac{1}{\mu(G)}\int_G S_{2k}(\gamma_g)d\mu(g)\right)=\frac{1}{\mu(G)}\int_G \tr (S_{2k})(\gamma_g)d\mu(g).\]
For almost every $g\in G$, we have $S_{2k}(\gamma_g)=\frac{v(g)^{\otimes {2k}}}{(2k)!}$ where $v(g)\in \mathfrak g$ is the unique vector satisfying $\exp(v(g))=g$. (See \eqref{eq:sig-simple}.) By taking the trace and use Lemma \ref{lem:tr}, we obtain
\[\tr (S_{2k}(\gamma_g))=\frac{||v(g)||^{2k}}{(2k)!}=\frac{d(e,g)^{2k}}{(2k)!}.\]
Hence by taking the integral over the normalized Haar measure, we have
\begin{align*}
    ||f||_{L^k}&=\left(\frac{1}{\mu(G)}\int_G d(e,g)^{2k}d\mu(g)\right)^{1/k}\\
    &=\left((2k)!\tr (A_{2k})\right)^{1/k}\\
    &=(\rtr (A_{2k}))^{1/k}.
\end{align*}

\end{proof}

\begin{corollary}\label{cor:diam}
    The diameter of $G$ can be recovered from $\tr\mathbb A(G)$. More precisely, we have
    \[\diam(G)=\lim_{k\rightarrow \infty}\left(\rtr (A_{2k})\right)^{1/2k}.\]
\end{corollary}
\begin{proof}
    We observe that $\diam(G)^2$ is the $L^\infty$-norm of $f(g)=d(e,g)^2$, hence the corollary follows immediately from the theorem together with the fact that $||f||_{L^\infty}=\lim_{k\rightarrow \infty}||f||_{L^k}$.
\end{proof}

\begin{theorem}\label{thm:metric-ball}
    The $\bar \mu$-measure of metric balls can be recovered from $\rtr\mathbb (A(G))$, that is, the function $F:[0, \diam(G)] \rightarrow \mathbb R_{\geq 0}$ given by
    \[F(R):= \bar \mu(B(R)) = \frac{\mu(B(R))}{\mu(G)}=\frac{\mu(\{g\in G:d(e,g)\leq R\})}{\mu (G)}\]
    can be recovered by $\rtr(\mathbb A(G))$.
\end{theorem}

\begin{proof}
   Let $D=\diam(G)^2$ and $\overline \mu=\frac{\mu}{\mu(G)}$ be the unique normalized Haar measure on $G$. We now consider the push forward measure $\nu:=f_*\overline \mu$, where $f(g) = d(e,g)^2$ as in Theorem \ref{thm:Lk-norm}. It is supported on $[0,D]$ and is absolutely continuous with respect to the Lebesgue measure. By Theorem \ref{thm:Lk-norm} it follows that
   \[\rtr (A_{2k}) =\int_G f(g)^k d\overline \mu(g)=\int_{[0,D]} x^k d\nu(x).\]
  for every $k\in \mathbb N$. Furthermore, for every $R \le \diam (G)$, the $\nu$-integration of the indicator function $1_{\{[0,R^2]\}}$ for the interval $[0,R^2]$  can be approximated by the $\nu$-integration of  a sequence of 
 polynomials $(p_\ell(x))_{\ell \in \mathbb{N}}$. In fact, one can find the coefficients $a^\ell_{N(\ell)}, \ldots, a^\ell_1, a^\ell_0$ in these polynomials $p_\ell(x) = a^\ell_{N(\ell)}x^{N(\ell)} + \ldots + a^\ell_1x + a^\ell_0$ explicitly (for instance, first performing convolution to $1_{\{[0,R^2]\}}$ against a smooth mollifier to obtain a sequence of $[0,1]$-valued continuous functions $(\mathbf{f}_\ell)_{\ell \ge 1}$ defined on $[0,D]$ which approximates $1_{\{[0,R^2]\}}$ pointwise, then by the dominated convergence theorem one has $ \int 1_{\{[0,R^2]\}} d\nu = \lim_{\ell \to \infty} \int \mathbf{f}_\ell d\nu$; next, for each $\ell$, picking a Bernstein polynomial $p_\ell$ for $\mathbf{f}_\ell$ such that $\|\mathbf{f}_\ell - p_\ell\|_\infty < 1/\ell$). All above implies that
 $$
 \nu([0,R^2]) = \lim_{\ell \to \infty} \int p_\ell(x) d\nu(x) = \lim_{\ell \to \infty} \sum_{k=0}^{N(\ell)}a^\ell_k \int x^k d\nu(x) = \lim_{\ell \to \infty} \sum_{k=0}^{N(\ell)}a^\ell_k \rtr (A_{2k}),
 $$
 i.e., $\nu([0,R^2])=F(R)$ can be recovered by $\rtr (\mathbb A(G))$.
\end{proof}

\begin{remark}
Note that the diameter of $G$ can also be recovered from the above relative volume function since
\[\diam(G)=\inf\{t\ge 0: F(t)=1\}.\]
\end{remark}

\begin{remark}\label{remark: moment problem}
 Recall that the average signature $\mathbb A(G)$ satisfies that
 $$
 \mathbb A(G) = \sum_{k=0}^\infty \int_G \frac{v(g)^{\otimes k}}{k!}d\bar \mu(g)
 $$
 is exactly the moment generating function of the random variable $v(\cdot)$ (which is the ``inverse'' of the exponential map on $G$) with respect to $\bar \mu$, see \eqref{eq: expression of Ak}. Therefore, what we are doing in the proof of Theorem \ref{thm:metric-ball}, which showed that the Haar probability measure $\bar \mu$ is uniquely determined by $\bar \tr(\mathbb A(G))$, can be viewed as solving the (Hausdorff) moment problem for $\bar \mu$ using the Stone-Weierstrass approximation theorem, see \cite{MomentProblem} for a comprehensive introduction to the moment problem, and see \cite{Chevyrev2022signaturekernel} and \cite{Chevyrev2016chf} for the moment problem of path-valued random variable.
\end{remark}

\begin{corollary}\label{cor:dim-vol-scalar} Let $G$ be a (connected) compact Lie group with a bi-invariant Riemannian metric. Given the rescaled trace spectrum of the average signature $\rtr(\mathbb A(G))\in \mathbb R^\infty$, then we can recover
\begin{enumerate}
    \item[$\bullet$] the dimension of $G$,
    \item[$\bullet$] the volume of $G$,
    \item[$\bullet$] the scalar curvature of $G$.
\end{enumerate}
    
\end{corollary}

\begin{proof}
We denote $w_n$ the volume of the $n$-dimensional unit ball in the standard Euclidean space. For any $n$-dimensional Riemannian manifold $M$, it is well-known that the metric ball $B(\epsilon)$ around a given point $p$ has the volume comparison (See \cite[Section 3.H.4]{GallotHulinLafontaine04})
\begin{equation}\label{eq:Taylor}
    \frac{\vol(B(\epsilon))}{w_n\epsilon^n}=1-\frac{S(p)}{6(n+2)}\epsilon^2+O(\epsilon^3),
\end{equation}
where $S(p)$ is the scalar curvature of $M$ at $p$ and $\vol(\cdot)$ is the Riemannian volume form on $M$. In our case, this Riemannian volume form induced by the bi-invariant metric is the Haar measure $\mu$. So for any measurable subset $A\subset G$, its volume is exactly given by $\vol(A)=\mu(A)$.

Since $F(\epsilon)$ can be recovered from $\rtr(\mathbb A(G))$ (by Theorem \ref{thm:metric-ball}), using the fact that 
\[1=\lim_{\epsilon\rightarrow 0^+}\frac{\mu(B(\epsilon))}{w_n\epsilon^n}=\lim_{\epsilon\rightarrow 0^+}\frac{F(\epsilon)\mu(G)}{w_n\epsilon^n}\]
we can also recover the dimension $n$ and the volume $\mu(G)$ from $\rtr(\mathbb{A}(G))$. More precisely, we have
\[n=\lim_{\epsilon\rightarrow 0^+}\frac{\log F(\epsilon)}{\log \epsilon},\]
and
\[\mu(G)=\lim_{\epsilon\rightarrow 0^+}\frac{w_n\epsilon^n}{F(\epsilon)}.\]
Furthermore, we note that by the left invariance of the metric, $G$ has constant scalar curvature. Then using the second order term in the above Taylor expansion \eqref{eq:Taylor}, we can recover the scalar curvature. More precisely, we have
\[S=\lim_{\epsilon\rightarrow 0^+}\frac{6(n+2)}{\epsilon^2}\left(1-\frac{F(\epsilon)\mu(G)}{w_n\epsilon^n}\right).\]
\end{proof}

\section{The product structure}\label{sec:product}
In Example \ref{ex:torus} and Example \ref{ex:trace}, we have already computed the average signature (and its trace) for the torus, which can be viewed as a Riemannian product of two circles. The main goal of this section is to generalize this phenomenon and analyze the average signature (and its trace) for a product of any two compact Lie groups, explicitly in terms of the average signatures (and their traces) of each factor.

First of all, we recall the shuffle product on the tensor algebra over a vector space. 
\begin{definition}
    Let $V$ be an $n$-dimensional vector space and $\{e_1,\dots,e_n\}$ be a basis of $V$. Given $e_{i_1}\otimes\cdots\otimes e_{i_k}\in V^{\otimes k}$ and $e_{j_1}\otimes\cdots\otimes e_{j_l}\in V^{\otimes l}$ with $i_1,\ldots,i_k, j_1,\ldots, j_l \in \{1,\ldots,n\}$, the shuffle product of the two elements is given by
    \[\left(e_{i_1}\otimes\cdots\otimes e_{i_k}\right)\shuffle \left(e_{j_1}\otimes\cdots\otimes e_{j_l}\right):=\sum_{\sigma} e_{\sigma(i_1)}\otimes\cdots\otimes e_{\sigma(i_k)}\otimes e_{\sigma(j_1)}\otimes\cdots\otimes e_{\sigma(j_l)},\]
    where the sum runs through all possible permutations $\sigma$ of the dummy index set (allow repeating) $\{i_1,\dots,i_k,j_1,\dots,j_l\}$ which preserves both the ordering of $\{i_1,\dots,i_k\}$ and $\{j_1,\dots,j_l\}$. The definition then naturally extends multi-linearly to the entire tensor algebra $T((V)) = \prod_{k=0}^\infty V^{\otimes k}$, which observes the commutativity and associativity laws.
\end{definition}

\begin{theorem}\label{thm:product}
    Let $G=G_1\times G_2$ be the Cartesian product of two (connected) compact Lie groups $G_1$ and $G_2$. Then for every $N \ge 0$,
    \[A_{N}(G)=\sum_{k=0}^N \frac{k!(N-k)!}{{N!}}\left(A_k(G_1)\shuffle A_{N-k}(G_2)\right),\]
    where $A_N(G)$ ($A_k(G_1), A_{N-k}(G_2)$ respectively) is the $N$-th ($k$-th, $(N-k)$-th) component of the average signature of $G$ ($G_1, G_2$) as in Definition \ref{def:average-signature}.
\end{theorem}

\begin{proof}
    We denote $\fg$ the Lie algebra of $G = G_1 \times G_2$, $\fg_1$ the Lie algebra of $G_1$ and $\fg_2$ the Lie algebra of $G_2$. It is well known that $\fg \cong \fg_1 \oplus \fg_2$.  For $i=1,2$ and $\mu_i$-almost all $g_i \in G_i$, we denote as usual $v(g_i) \in \fg_i$ the unique vector such that $\gamma_{g_i}(t) = \exp(tv(g_i))$ ($t \in [0,1]$) is the unique shortest constant speed geodesic connecting the neutral element $e_i \in G_i$ to $g_i$. 
    
    Recall that for a Riemannian product $G=G_1\times G_2$, a curve $\gamma(t)=(\gamma_1(t),\gamma_2(t))$ is a (constant speed) geodesic in $G$ if and only if both $\gamma_i(t)$ are constant speed geodesics in $G_i$. This is due to the fact that the Levi-Civita connection $\nabla$ on $G$ relates with those $\nabla^{(i)}$ on $G_i$ simply by $\nabla= \nabla^{(1)}+\nabla^{(2)}$. It follows that
    \[\nabla_{\gamma'(t)}\gamma'(t)=\left(\nabla^{(1)}_{\gamma_1'(t)}\gamma_1'(t),\nabla^{(2)}_{\gamma_2'(t)}\gamma_2'(t)\right)=0\]
    if and only if both $\nabla^{(i)}_{\gamma_i'(t)}\gamma_i'(t)=0$.
    
    Apply to our context, we know that the curve $\gamma_g(t)=\exp(t(v(g_1)+v(g_2))), t\in [0,1])$ is a geodesic segment connecting the neutral element $e$ to $g$ in $G$. We claim it is also unique length minimizing. Suppose $\eta_g(t),t\in[0,1]$ is another (constant speed) geodesic path which connects $e$ to $g$ in $G$, then by writing $\eta_g(t)=(\eta_1(t),\eta_2(t))$, we see that each $\eta_i$ is a (constant speed) geodesic path connecting $e$ to $g_i$ in $G_i$. But since $\gamma_{g_i}$ is unique length minimizing, we have $||\eta_i||_{G_i}> ||\gamma_{g_i}||_{G_i}=||v(g_i)||_{G_i}$. This shows that
    \[L(\eta_g)= \sqrt{||\eta_1'||_{G_1}^2+||\eta'_2||_{G_2}^2}>\sqrt{||v(g_1)||_{G_1}^2+||v(g_2)||_{G_2}^2}=L(\gamma_g),\]
    hence $\gamma_g$ must be the unique length minimizing geodesic connecting $e$ to $g$ in $G$. It follows that, for $\mu: = \mu_1 \times \mu_2$-almost all $g=(g_1,g_2)\in G = G_1\times G_2$, the unique vector $v(g)\in \mathfrak g$ such that $\gamma_g(t)=\exp(tv(g))$ $(t\in [0,1])$ is the length minimizing (constant speed) geodesic connecting the neutral element $e \in G$ to $g$ must satisfy $v(g) = v(g_1) + v(g_2)$ under the splitting $\mathfrak g\cong \mathfrak g_1\oplus \mathfrak g_2$.
    
    Let $\bar \mu$, $\bar \mu_1$ and $\bar \mu_2$ be the normalized probability Haar measure on $G, G_1, G_2$ and we have $\bar \mu = \bar \mu_1 \times \bar \mu_2$. Then by Fubini's theorem (noting that the shuffle product is bilinear) together with the definition of shuffle product, we have that for each $N \ge 0$,
    \begin{align*}
        N!A_N(G)&=\int_G v(g)^{\otimes N}d\bar \mu(g)\\
        &=\int_{G_1}\int_{G_2} (v(g_1)+v(g_2))^{\otimes N}d\bar \mu_1(g_1)d\bar \mu_2(g_2)\\
        &=\sum_{k=0}^N \int_{G_1}\int_{G_2} v(g_1)^{\otimes k}\shuffle v(g_2)^{\otimes {(N-k)}} d\bar \mu_1(g_1)d\bar \mu_2(g_2)\\
        &=\sum_{k=0}^N\left(\int_{G_1}v(g_1)^{\otimes k}d\bar \mu_1(g_1)\right)\shuffle\left(\int_{G_2}v(g_2)^{\otimes {(N-k)}}d\bar \mu_2(g_2)\right)\\
        &=\sum_{k=0}^N k!(N-k)!\left(A_k(G_1)\shuffle A_{N-k}(G_2)\right).
    \end{align*}
    Therefore the theorem holds.
\end{proof}


Note that for Theorem \ref{thm:product} we actually do not need the Riemannian structure (i.e., the choice of bi-invariant Riemannian metric) on the Lie groups, because the geodesics joining $e$ and $g$ for $g$ outside the cut locus of $e$ are always of the form $t \mapsto tv(g)$ for a unique $v(g) \in \fg$ regardless of the choice bi-invariant metric, 
see our discussions in Section \ref{sec:average-signature}. However, to compute the (rescaled) trace of the average signature on $G = G_1 \times G_2$, we have to fix bi-invariant metrics on $G_1$ and $G_2$, as the notion of trace operator (see Definitions \ref{def:tr}, \ref{def:rtr}) depends on the choice of bi-invariant metrics.

\begin{corollary}
    Let $G_1$ and $G_2$ be two connected, compact Lie groups endowed with the bi-invariant metrics $\langle \cdot, \cdot \rangle_{G_1}$ and $\langle \cdot, \cdot \rangle_{G_2}$, and let $G=G_1\times G_2$ be the Riemannian Cartesian product of $G_1$ and $G_2$, that is, its bi-invariant metric $\langle \cdot, \cdot \rangle_{G}$ is given by the product metric of $\langle \cdot, \cdot \rangle_{G_1}$ and $\langle \cdot, \cdot \rangle_{G_2}$. Then associated to these metrics, we have
    \[\rtr(A_{2N}(G))=\sum_{k=0}^N {N\choose k}\rtr(A_{2k})(G_1) \rtr (A_{2N-2k}(G_2))\]
    holds for all $N \in \mathbb N$.
\end{corollary}
\begin{proof}
The proof follows straightforward from Theorem \ref{thm:product} by taking the trace on both sides. The key ingredient is that
\[\left(e_{i_1}\otimes e_{i_1}\otimes\cdots \otimes e_{i_k}\otimes e_{i_k}\right)\shuffle \left(\overline e_{j_1}\otimes \overline e_{j_1}\otimes\cdots \otimes \overline e_{j_l}\otimes \overline e_{j_l}\right)\]
can only contribute to the trace for the total amount of ${k+l \choose k}$ shuffles which preserves each adjacent pairs $e_{i_p}\otimes e_{i_p}$ and $\overline e_{j_q}\otimes \overline e_{j_q}$. However, we present a more enlightening geometric proof here.

For each Lie group $G, G_1, G_2$, each bi-invariant metric will induce a volume form which gives a Haar measure on each Lie group. As usual, we denote $\bar \mu, \bar \mu_1, \bar \mu_2$ the normalized Haar measure on $G, G_1, G_2$ respectively, and we have $\bar \mu=\bar \mu_1\times \bar\mu_2$.  We also denote $\mathbb A(G_1)=\sum_{k=0}^{\infty} A_{2k}$, $\mathbb A(G_2)=\sum_{k=0}^{\infty} B_{2k}$ and $\mathbb A(G)=\sum_{k=0}^{\infty} C_{2k}$ the average signature of $G_1$, $G_2$ and $G$ respectively. Let $d_{G_1}$, $d_{G_2}$ and $d_G$ denote the Riemannian distance on $G_1$, $G_2$ and $G$ (relative to their equipped bi-invariant metrics) respectively. Then by Theorem \ref{thm:Lk-norm} and apply the binomial theorem and the Fubini's theorem, we have for each $N \in \mathbb N$
\begin{align*}
 \rtr (C_{2N}) &=\int_G d_G(e,g)^{2N} d\bar\mu(g)\\
 &=\int_{G_1}\int_{G_2} (d_{G_1}(e_1,g_1)^2+d_{G_2}(e_2,g_2)^2)^N d\bar \mu_1(g_1)d\bar \mu_2(g_2)\\
 &=\sum_{k=0}^N  {N\choose{k}}\left(\int_{G_1} d_{G_1}(e_1,g_1)^{2k}d\bar \mu_1(g_1)\right)\left(\int_{G_2} d_{G_2}(e_2,g_2)^{2N-2k}d\bar \mu_2(g_2)\right)\\
 &=\sum_{k=0}^N {N\choose k} \rtr(A_{2k}) \rtr (B_{2N-2k}).
\end{align*}
\end{proof}


\bibliographystyle{alpha}
\bibliography{myref}
\end{document}